\documentclass[12pt,reqno]{amsart}
\usepackage{amssymb,amsmath,amsthm,amsfonts}
\usepackage{color}
\usepackage{a4wide}
\usepackage{pxfonts}
\usepackage[OT2,T1]{fontenc}
\usepackage{pifont}

\newtheorem{prop}{Proposition}[section]

\newtheorem{thm}[prop]{Theorem}
\newtheorem*{thm*}{Theorem}
\newtheorem{lem}[prop]{Lemma}

\newtheorem{cor}[prop]{Corollary}
\newtheorem{rem}[prop]{Remark}
\theoremstyle{definition}
\newtheorem{exam}[prop]{Example}

\newcommand{\dst}{\displaystyle}

\newcommand{\NN}{\mathbb{N}}
\newcommand{\RR}{\mathbb{R}}
\newcommand{\CC}{\mathbb{C}}
\newcommand{\C}{\mathbb{C}}

\newcommand{\Z}{\mathbb{Z}}

\newcommand{\cS }{{\mathcal{S}}}

\newcommand{\cA}{{\mathcal{A}}}
\newcommand{\cK}{{\mathcal{K}}}
\newcommand{\wl}{\widetilde{\Lambda}}
\newcommand{\wf}{\widetilde{F}}

\newcommand{\ahc}{{\mathcal{F}_{\varphi}^\infty}}

\newcommand{\ahd}{{\mathcal{F}_{\varphi}^2}}
\newcommand{\ahp}{{\mathcal{F}_{\varphi}^p}}

\newcommand{\Card}{\operatorname{Card}}
\newcommand{\dist}{\operatorname{dist}}

\newcommand{\Hol}{\mathrm{Hol}}

\newcommand{\dl}{\delta_{nm}}

\DeclareMathOperator{\K}{\Bbbk}

\DeclareMathOperator{\bk}{\textbf{k}}

\renewcommand{\phi}{\varphi}

\newcommand{\nc}{\newcommand}
\nc{\Log}{\operatorname{Log}}
\nc{\eps}{\varepsilon}
\nc{\lra}{\longrightarrow}

\begin{document}
\title[\sf Sampling, interpolation and Riesz Bases in Small Fock spaces]{Sampling, interpolation and Riesz bases in  small  Fock spaces}
\author[Baranov, Dumont, Hartmann, Kellay]{A. Baranov, A. Dumont, A. Hartmann, K. Kellay}
\address{A. Baranov\\ Department of Mathematics and Mechanics\\ St. Petersburg State University\\
St. Petersburg\\ Russia}
\email{anton.d.baranov@gmail.com}
\address{A. Dumont\\ LATP\\CMI\\Aix--Marseille Universit\'e\\
39, rue F. Joliot-Curie\\13453 Marseille\\France}
\email{dumont@cmi.univ-mrs.fr}

\address{A. Hartmann \& K. Kellay \\IMB\\Universit\'e Bordeaux I\\
351 cours de la Liberation\\33405 Talence \\France}
\email{Andreas.Hartmann@math.u-bordeaux1.fr}
\email{kkellay@math.u-bordeaux1.fr}


\keywords{Sampling, interpolation, Riesz bases, small Fock spaces,
de Branges spaces}
\subjclass[2000]{Primary 30H05; Secondary 30D10, 30E05}
\thanks{The first author was supported by the Chebyshev Laboratory (St. Petersburg State University) under RF Government grant 11.G34.31.0026 and by JSC Gazprom Neft, by Dmitry Zimin's Dynasty Foundation and by RFBR grant14-01-00748. This work was partially supported by ANR-09-BLAN-0058-02}
\begin{abstract} { We give a complete description of Riesz bases
of reproducing kernels in small Fock spaces. This characterization is in the spirit of the 
well known Kadets--Ingham 1/4 theorem for Paley--Wiener spaces.
Contrarily to the situation in Paley--Wiener spaces, a link can be established
between Riesz bases in the Hilbert case and corresponding complete interpolating sequences
in small Fock spaces with associated uniform norm.
These results allow to show that if a sequence has a density stricly different from
the critical one  then either it can be completed or reduced to a complete interpolating 
sequence.
In particular, this allows to give necessary and sufficient conditions for
interpolation or sampling in terms of densities.}
\end{abstract}

\maketitle


\section{Introduction and main results.}\label{section1}

Interpolation and sampling problems are well studied objects.
Complete results for corresponding sequences are known for broad classes
of spaces of analytic functions. We refer the reader to the monograph by Seip for
an account on these problems \cite{S}. 
Two prominent examples here are the Fock spaces and the Bergman spaces.
For these, interpolating and sampling sequences have been studied by
Seip in the classical situation. 
More general weights have been discussed in the 1990s by Berndtsson and Ortega-Cerd\`a \cite{BO},
Lyubarskii and Seip \cite{LS}, and later Marco, Massaneda and Ortega-Cerd\`a \cite{MMO}. 
More recently, a series of results was obtained for 
"small"\, (i.e., with slowly growing weights) versions of 
these spaces. In this case the geometric properties 
of sampling/interpolating sequences change significantly. Seip showed that 
in small Bergman spaces, locally,
interpolating sequences look like interpolating sequences in Hardy
spaces \cite{S0}. In small Fock spaces, Borichev and
Lyubarskii recently exhibited Riesz bases of reproducing kernels 
\cite{BL}. In this paper we will investigate
further the situation of small Fock spaces $\ahp$ for  the weight
$\varphi(z) = \alpha(\log^+ r)^2$.
We focus on the Hilbert situation $p=2$ and on the case $p = \infty$. 
As it turns out, no density characterization can be expected 
for interpolation or sampling. There are actually sequences
which are simultaneously interpolating 
and sampling, also called {\it complete interpolating sequences}.
Note that complete interpolating sequences necessarily have  critical density
(as defined below). We also provide sequences with critical density which are neither
interpolating nor sampling for $p=2,\infty$. 

The central result of this paper is a characterization of complete interpolating
sequences when $p=2$ which is in the spirit of the famous $1/4$ Kadets--Ingham theorem 
in the Paley--Wiener space and its more general version due to 
Avdonin \cite{avd}. A different characterization using in particular a Muckenhoupt type 
condition and based on different techniques is
discussed by Belov, Mengestie and Seip for more general spaces in \cite{BMS}.
The novelty of our approach is the use of a rather elementary Hilbert space result,
namely Bari's Theorem, adapted to our situation.
Moreover, we will show  that, surprisingly, the same characterization 
applies to complete interpolating sequences in $\ahc$, the $L^\infty$-counterpart 
of the Fock space $\ahd$. 
From our characterization, we will also deduce sufficient density
conditions for interpolation and sampling.

We also would like to emphasize the connection between these
spaces and the de Branges spaces \cite{br}. The complete interpolating
sequence introduced by Borichev and Lyubarskii in \cite{BL}
defines a generating function $G$ which, when $p=2$, allows to identify 
the Fock spaces we are interested in with the de Branges space $\mathcal{H}(G)$. 
Consequently, the measure $dx/|G(x)|^2$ is a sampling measure for our Fock spaces, which
answers a question raised in \cite{M}. Note that in \cite{MNO}, the authors
consider sampling and interpolation in the class of de Branges spaces for
which the phase function defines a doubling measure. Our space corresponds
to the situation when the phase function is locally but not globally doubling  
so that their results apparently do not apply here. Still it can be observed that these
authors obtain a similar kind of density characterization as ours when $p=2$
(at least for real sequences they consider).


\subsection{Definition of small Fock spaces}
We now introduce the necessary notation.
Let $\varphi(z) = \alpha(\log^+|z|)^2$, which is 
a subharmonic radial function
with $\varphi(r)\nearrow +\infty$,  $r\to +\infty$,
and define the associated Fock space,
$$
\ahd=\big\{f\in \Hol (\C)\text{ : } \|f\|_{\varphi ,2}^{2}:=\int_\C|f(z)|^2e^{-2 
\varphi(z)}dm(z) <\infty\big\}.
$$
In \cite{BL} Borichev and Lyubarski have shown the existence of 
{\it complete interpolating} sequences in $\ahd$ (i.e., simultaneously 
interpolating and sampling for $\ahd$,
see precise definitions below).

The sequence they introduced will be the reference sequence for our considerations: 
\begin{equation}
\label{bls}
 \Gamma=\Gamma_\alpha=\{e^{\frac{n+1}{2\alpha}}e^{i\theta_n}\}_{n\ge 0},\qquad 
 \theta_n\in \RR.
\end{equation}

In order to define sampling and interpolating sequences for $\ahd$, 
we consider first $\bk_z$, the reproducing kernel of $\ahd$:
$$
\langle f, \bk_z\rangle_{\ahd}=f(z),\qquad f\in\ahd,\quad z\in\mathbb C.
$$
According to \cite[Lemma 2.7]{BL}, the kernel admits the following estimate: 
\begin{equation}
\label{kernel-estimate}
 \|\bk_z\|^2_{\varphi,2}=\bk_z(z) \asymp \frac{e^{2\varphi(z)}}{1+|z|^2},\qquad z\in \C.
\end{equation}

The sequence $\Lambda\subset \C$  is called {\it sampling} for $\ahd$ if
\[
 \|f\|_{\varphi,2}^{2}\asymp \|f\|_{\varphi,2,\Lambda}^{2}:=\sum_{\lambda\in  
 \Lambda}\frac{|f(\lambda)|^2}{\bk_\lambda(\lambda)}, 
 \quad f\in\ahd,
\]
and {\it interpolating} 
if for every $v=(v_\lambda)_{\lambda\in \Lambda}\in \ell^2_{\varphi,\Lambda}$ , i.e., 
 $\|v\|_{\varphi,2,\Lambda}<\infty$, there exists $f\in \ahd$ such that
$$ 
   v=f|_\Lambda.
$$
Let $\K_\lambda=\bk_\lambda/\|\bk_\lambda\|_{\varphi,2}$ 
be the normalized reproducing kernel at $\lambda$. Let $\Lambda\subset \CC$. 
We say that $\{\K_\lambda\}_{\lambda\in \Lambda}$ is a {\it Riesz sequence} in $\ahd$ if 
for some $C>0$ and  for each finite sequence $\{a_\lambda\}$, we have
$$
 \frac{1}{C}\sum_{\lambda\in \Lambda}|a_\lambda|^2\leq \big\| 
 \sum_{\lambda\in \Lambda} a_\lambda\K_\lambda \big\|_{2,\varphi}^{2}
 \leq C\sum_{\lambda\in \Lambda}|a_\lambda|^2,
$$
and a {\it Riesz basis} if it is also complete. It is well known that $\Lambda$
is interpolating if and only if $\{\K_\lambda\}_{\lambda\in \Lambda}$ is a Riesz
sequence, and $\Lambda$ is complete interpolating if and only if 
$\{\K_\lambda\}_{\lambda\in \Lambda}$ is a Riesz basis in $\ahd$.

It should be mentioned that our case $\varphi(r) = \alpha (\log^+ r)^2$
corresponds to the critical growth of the weight for which Riesz bases of reproducing kernels exist. 
Recall that, by the results of Seip and Seip--Wallst\'en \cite{S1,S2} 
there are no complete interpolating sequences
for the classical Fock space (i.e., for $\phi(r)=r^2$). 
For more general (in particular, rapidly growing) weights the same 
was shown in \cite{BDK,MMO,LS} (see also \cite{S} as a general source). Some examples 
of slowly growing weights such that no complete interpolating sequences exist
were given in \cite{yul}. Finally, Borichev and Lyubarskii
\cite[Theorem 2.5]{BL} have shown that, under some regularity conditions,
if $(\log^+ r)^2\ll \varphi(r)\ll r^2$, then the corresponding  Fock space 
does not possess a Riesz basis of reproducing kernels. 
We mention that there is no known weight for Bergman spaces for which there
are Riesz bases of reproducing kernels.


\subsection{Description of complete interpolating sequences in $\ahd$}
Our central result is a characterization 
of complete interpolating sequences
in terms of their deviation from the sequence $\Gamma=\Gamma_\alpha$ 
defined in \eqref{bls}.

Before stating the theorem we need to define separation as in \cite{MMO,BL}.
Set
$$
\rho(z)=(\Delta \varphi(z))^{-1/2},
$$
where $\Delta \varphi(r)=\varphi''(r)+{\varphi'(r)}/{r}$, $r>0$. 
We associate with $\rho$ a ``distance'' (a semi-metric):
$$d_\rho(z,w)=\frac{|z-w|}{1+\min(\rho(z),\rho(w))},\qquad z,w\in \C.$$
Note that when $z,w$ are in a fixed disk, this distance is comparable
to Euclidean distance.

The sequence $\Lambda\subset \C$ is said to be $d_\rho$-separated if there is 
$d_\Lambda>0$ such that
\begin{eqnarray*}
 \inf \{d_\rho(z,w)\text{ : } z,w\in \Lambda\text{, } z\neq w\}\ge d_\Lambda.
\end{eqnarray*}

In the specific situation $\varphi(r)=\alpha (\log^+ r)^2$, we have $\rho(r)=  r/\sqrt{2\alpha}$,
$r\ge 1$.
Hence
$$
d_\rho(z,w)={\sqrt{2\alpha}} \frac{|z-w|}{\sqrt{2\alpha}+\min(|z|,|w|)}.
$$ 
In particular, for $\beta<1$ and $0\neq \lambda\in\C$ the ball 
corresponding to this distance is given by
\[
 D_{\rho}(\lambda,\beta):=\{z\in \C:d_{\rho}(z,\lambda)<\beta\}.
\]
When $\beta$ is small, $D_{\rho}(\lambda,\beta)$ is comparable to
a Euclidean disk $D(\lambda,q|\lambda|)$ with a suitable constant $q$ depending
on $\beta$.
From this 
we deduce that $\Lambda$ is $d_{\rho}$-separated
if and only if there
exists $c>0$ such that the Euclidean disks $D(\lambda,c|\lambda|)$,
$\lambda\in\Lambda$, are disjoint.

   \begin{thm}\label{BaseRiesz}  
Let $\alpha>0$, $\varphi(r)=\alpha(\log^+ r)^2$,   
let $\Gamma=\{\gamma_n\}_{n\geq 0}=\{e^\frac{n+1}{2\alpha}\}_{n\geq 0}$ and
let $\Lambda=\{\lambda_n\}_{n\ge 0}$ 
with $\lambda_n=\gamma_ne^{\delta_n} e^{i\theta_n}$, 
$|\lambda_n|\leq |\lambda_{n+1}|$, 
$\theta_n\in \RR$. Then $\{\K_\lambda\}_{\lambda\in \Lambda}$ is a Riesz basis
for $\ahd$ if and only if the following three conditions hold.
   \begin{enumerate}
   \item[(a)] $\Lambda$ is $d_\rho$-separated,
   \item[(b)] $(\delta_n)\in \ell^\infty$,
   \item[(c)] there exists $N\geq 1$ and $\delta>0$ such that
   $$\sup_n\frac{1}{N}\Big|\sum_{k=n+1}^{n+N}\delta_k\Big|\leq \delta<\frac{1}{4\alpha}.$$
   \end{enumerate}
   \end{thm}

For $N=1$, condition (c) of Theorem \ref{BaseRiesz} becomes $\sup_n|\delta_n| <1/4$,
which resembles a well-known stability 
result for complete interpolating sequences in the Paley--Wiener space 
(equivalently, Riesz bases of exponentials) -- the famous 1/4--Theorem
due to Ingham and Kadets (see, e.g., \cite{HNP, N}). Also, 
for arbitrary $N$, Avdonin considers the sufficiency of that condition 
in the Paley--Wiener space \cite{avd}. However, there is an essential 
difference since complete interpolating sequences in the Paley--Wiener 
space can not be described in terms of perturbations and more subtle 
characteristics (e.g., the Muckenhoupt condition) appear \cite{HNP}, 
while in the case of spaces of very slow growth such a characterization 
turns out to be possible. 
As already mentioned, very close results were obtained by Belov, Mengestie 
and Seip in \cite{BMS} where the boundedness and invertibility problem of 
a discrete Hilbert transform on lacunary sequences was solved. 
Though Theorem \ref{BaseRiesz} is not formally covered by the results stated in 
\cite{BMS}
(see Remark  \ref{bsm}), 
it seems that one can obtain our characterization  
using the methods of that paper. However, as mentioned earlier, 
our proof being based on Bari's theorem is essentially elementary. 

As an immediate consequence of Theorem \ref{BaseRiesz}, the reader should note that 
if a sequence $\{\lambda_n\}$ is complete interpolating
then any sequence $\{\lambda_n e^{i\theta_n}\}$ will be complete interpolating
(as is the case for $\Gamma_{\alpha}$).


\subsection{Weighted Fock spaces with uniform norm}
We also deal with the case $p=\infty$.
The corresponding weighted Fock space is defined by
$$
 \ahc=\big\{f\in \Hol (\C)\text{ : } \|f\|_{\varphi ,\infty}
 :=\sup_{z\in \C}|f(z)|e^{- \varphi(z)}<\infty\big\}.
$$
A sequence $\Lambda\subset \C$ is called {\it sampling} for $\ahc$,  
if there exists  $L>0$ such that
$$
\|f\|_{\varphi,\infty} \leq L \|f\|_{\varphi,\infty,\Lambda}:=
 L\sup_{\lambda\in \Lambda}|f(\lambda)|e^{- \varphi(\lambda)} , \qquad f\in\ahc.
$$

A sequence $\Lambda$ is called {\it interpolating} for $\ahc$ if for every sequence
$v=(v_\lambda)_{\lambda\in \Lambda}$ in $\ell^{\infty}_{\varphi,\Lambda}$,
i.e., such that $\|v\|_{\varphi,\infty,\Lambda}<\infty$,
there is a function $f\in \ahc$ such that $v=f|_\Lambda$.

It turns out that there exist complete interpolating sequences for $\ahc$
(i.e., as in $\ahd$, simultaneously interpolating and sampling sequences for
$\ahc$), 
which differs in this respect from most other known spaces of entire functions (e.g., 
complete interpolating sequences exist for the Paley--Wiener space $PW^p_a$ 
for $1<p<\infty$ \cite{ls1}, but not for $p=\infty$).

Our second main result is as follows:

\begin{thm} 
\label{infin}
 Let $\varphi(r)=\alpha(\log^+ r)^2$,  $\alpha>0$. Then
 a sequence $\Lambda$ is a complete interpolating sequence 
 for $\ahc$ if and only if for some \textup(any\textup) $\lambda\in \Lambda$ the sequence
 $\Lambda \setminus \{\lambda\}$ is a complete interpolating sequence 
 for $\ahd$.
\end{thm}

Thus, any complete interpolating sequence for $\ahc$ is a small (in the sence of 
conditions (a)--(c) of Theorem \ref{BaseRiesz}) perturbation 
of the sequence $\widetilde {\Gamma} = \Gamma \cup\{1\}$.
\\

We derive from Theorems \ref{BaseRiesz} and \ref{infin}
several density conditions for interpolation
and sampling in $\ahp$, $p=2,\infty$.
Before stating these results, we need some more notation.

Let $\cA(r,R)$ be the annulus centered at the origin with inner and outer radii $r$ and $R$: $\cA(r,R):=\{z\in \C\text{ : } r\le |z|<R\}$. For a $d_{\rho}$-separated
sequence $\Lambda$ we define the lower and upper densities respectively by
$$D^{-}({\Lambda})=\liminf_{R\to +\infty}\liminf_{r\to+\infty}\frac{\Card ({\Lambda}\cap \cA(r,Rr))}{\log R}$$
and
$$
D^{+}({\Lambda})=\limsup_{R\to +\infty}\limsup_{r\to+\infty}
\frac{\Card ({\Lambda}\cap \cA(r,Rr))}{\log R}.
$$

These densities do not change when we remove or add a finite number
of points to $\Lambda$. 

\subsection{Density criteria for sampling and interpolation}

As an application of our result on Riesz bases Theorem \ref{BaseRiesz} (as well as Theorem
\ref{infin}) we can show that
each set with $D^+(\Lambda)<2\alpha$ (respectively $D^-(\Lambda)>2\alpha$) can be 
completed (reduced) to a complete interpolating sequence. More precisely: 

\begin{thm}\label{denscomp}
 Let $\varphi(r) = \alpha(\log^+r)^2$, let $p=2, \infty$, and let $\Lambda$ be 
 a $d_\rho$-separated sequence. Then

 {\rm (i)} if $D^+(\Lambda) < 2\alpha$, then $\Lambda$ is a subset of some complete 
 interpolating sequence in $\ahp$;

 {\rm (ii)} if  $D^-(\Lambda) > 2\alpha$, then $\Lambda$ contains a complete 
 interpolating sequence in $\ahp$. 
\end{thm}

This result, together with a classical comparison method by Ramanathan-Steger \cite{MMO,OS}, 
allows us to deduce our density results.

\begin{thm} \textup(Sampling, $p=\infty$\textup)
 \label{thm1} 
 Let $\varphi(r)=\alpha(\log^+ r)^2$,  $\alpha>0$. Then 

 \begin{itemize}
 \item[(i)] every $d_\rho$-separated sequence $\Lambda$ 
 with $D^{-}({\Lambda})> 2\alpha$,  
 is a set of  sampling for $\ahc$; 
 
 \item[(ii)] if the sequence $\Lambda$ is a set of sampling for $\ahc$ then it contains a $d_{\rho}$-separated subsequence
 $\widetilde{\Lambda}$  with $D^{-}(\widetilde{\Lambda})\geq 2\alpha$.
 \end{itemize}
\end{thm}

\begin{thm} \textup(Sampling, $p=2$\textup) 
 \label{thm2}
 Let $\varphi(r)=\alpha(\log^+ r)^2$,  $\alpha>0$. Then

 \begin{itemize}
 \item[(i)] every $d_\rho$-separated sequence $\Lambda$ 
 with $D^{-}({\Lambda})> 2\alpha$, is a set of  sampling for $\ahd$;
 
 \item[(ii)] if the sequence $\Lambda$ is a set of sampling for $\ahd$, 
 then it is a finite union of  $d_{\rho}$-separated subsequences and 
 $\Lambda$ contains a $d_\rho$-separated sequence 
 $\widetilde{\Lambda}$  such that  $D^{-}(\widetilde{\Lambda})\geq 2\alpha$.
 \end{itemize}
\end{thm}
 
\begin{thm} \textup(Interpolation, $p=2,\infty$\textup)
 \label{thm-int}
 Let $\varphi(r)=\alpha(\log^+ r)^2$,  $\alpha>0$. Then

\begin{itemize}
 \item[(i)] every $d_\rho$-separated sequence $\Lambda$ 
 with $D^{+}({\Lambda})< 2\alpha$
 is a set of interpolation for $\ahp$, $p=2,\infty$;

 \item[(ii)] if the sequence $\Lambda$ is a set of interpolation for $\ahp$, $p=2,\infty$, 
 then it is a $d_{\rho}$-separated sequence with $D^{+}({\Lambda})\leq 2\alpha$.
 \end{itemize}
\end{thm}

In the case when the density is critical, i.e., 
$D^+(\Lambda)= D^-(\Lambda)=2\alpha$, any of the following situations may occur: 
a system my be complete interpolating, either complete or interpolating, and, finally,
neither complete nor interpolating (see Section \ref{section9}
for the corresponding examples). Thus, there are no 
density characterizations for sampling or interpolating sequences 
for $p=2,\infty$. 

The paper is organized as follows. In the next section, we present some 
elementary results on sampling and interpolation in our spaces. Some of 
them follow from a Bernstein type inequality that we will also give
in this section. It is an interesting remark that we can 
consider our Fock spaces as subspaces of a suitable $H^{\infty}$ from 
which we deduce that half-lines are sampling for $p=\infty$. 
Furthermore, we show that the lower density of
a zero sequence has to be less than or equal to the critical density. 
Sections \ref{section7}  and \ref{section8} are devoted to the proof of our main results
(Theorems \ref{BaseRiesz} and \ref{infin}), from which 
we deduce the density results on sampling and interpolation in Section \ref{section4}.
Examples of sequences of critical density which are neither sampling 
nor interpolating are discussed in Section \ref{section9}. There we will also
show that it is not possible to switch from an interpolating sequence to a sampling
sequence by adding one point without one of the sequences being complete interpolating.
\\

A final word on notation:
$A\lesssim B$ means that there is a constant $C$ independent of the
relevant variables such that $A \leq CB$. We write $A\asymp B$  
if both $A\lesssim B$ and $B\lesssim A$.

\subsection*{Acknowledgements} Alexander Borichev has read the first draft of the paper
and suggested many simplifications and improvements.
The authors are grateful to him and also to  
Yurii Belov,  Yuri Lyubarskii, 
Pascal Thomas and  Kristian  Seip, for helpful discussions.   


\section{Preliminary results}\label{section2}
\subsection{$d_\rho$-separated sequences}
Recall that $\varphi(r)=\alpha (\log^+ r)^2$, and
$$
d_\rho(z,w)={\sqrt{2\alpha}} \frac{|z-w|}{\sqrt{2\alpha}+\min(|z|,|w|)}. 
$$

A central tool in our discussion is 
the following Bernstein type result whose proof can be found in
\cite[Lemma 4.1]{BDK}.  
\begin{lem}
\label{lem1}
Let $f$ be a holomorphic function. 

\begin{enumerate}
\item[(i)] If $\|f\|_{\varphi,\infty}=1$, then for every $c>0$ there exists 
$0<\beta<1$ such that whenever $|f(z_0)|e^{-\varphi(z_0)}\ge c$
for some $z_0\in\C$, then
for every $z\in D_{\rho}(z_0,\beta)$ we have
$$|f(z)|e^{- \varphi (z)}\ge \frac{c}{2}e^{-\alpha\pi^2}.$$

\item[(ii)] There is $\beta_0>0$, such that if $0<\beta<\beta_0$, then for $z\in\C$ with $d_{\rho}(z,z_0)\le \beta$ we have
$$||f(z)|e^{- \varphi (z)}-|f(z_0)|e^{-\varphi(z_0)}|\lesssim d_\rho(z,z_0)\max_{D_\rho(z_0,\beta)}|f|e^{-\varphi}.$$
\item[(iii)]  $\displaystyle |f(z)|e^{-\varphi(z)}\lesssim\frac{1}{|z|^2}\int_{D_{\rho}(z,\beta)}
 |f(w)|e^{-\varphi(w)}dm(w).$
\end{enumerate}
\end{lem}


From Lemma \ref{lem1} we can deduce the following immediate corollaries
(proofs can be found for instance in \cite[Chapter 5]{HKZ} or \cite[Chapter 4]{Z}).

\begin{cor}\label{lemma2} If $\Lambda$ is sampling for $\ahp$, $p=2,\infty$,  then
there exists a $d_{\rho}$-separated sequence $\widetilde{\Lambda}\subset \Lambda$
which is sampling for $\ahp$.
\end{cor}

\begin{cor}\label{lemma2x} 
Every set of interpolation for $\ahp$, $p=2,\infty$, is $d_\rho$-separated.
\end{cor}

We also need a uniform control of the sampling constant for small
perturbations.

\begin{cor}\label{pertp=2}
Let $\Lambda=\{\lambda_n\}$ be a separated sampling sequence.
Then there is a $\delta>0$ and $C>0$ depending only on $\delta$ such that
for every $\widetilde{\Lambda}=\{\tilde{\lambda_n}\}$ with
$d_{\rho}(\lambda_n,\tilde{\lambda}_n)\le \delta$,
we have
\[
 \frac{1}{C}\|f\|_{2,\varphi}
 \lesssim \| f|_{\widetilde{\Lambda}} \|_{2,\varphi,\widetilde{\Lambda}}.
\]
\end{cor}

\begin{proof}
First note that when $d_{\rho}(\lambda_n,\tilde{\lambda}_n)\le \delta$,
then $1+|\lambda_n|^2\asymp 1+|\tilde{\lambda}_n|^2$ (constants are uniform for
$\delta\le\delta_0$).
Now
\begin{eqnarray*}
\lefteqn{\|f\|_{2,\varphi}^2\asymp\sum_n(1+|{\lambda}_n|^2)
 |f({\lambda}_n)|^2e^{-2\varphi({\lambda_n})}}\\
 && \le \sum_n(1+|{\lambda}_n|^2)
 |f(\tilde{\lambda}_n)|^2e^{-2\varphi(\tilde{\lambda_n})}
 +\sum_n(1+|{\lambda}_n|^2)
 \left||f(\tilde{\lambda}_n)|^2e^{-2\varphi(\tilde{\lambda_n})}
 - |f({\lambda}_n)|^2e^{-2\varphi({\lambda_n})}\right|\\
 &&\lesssim \sum_n(1+|\tilde{\lambda}_n|^2)
 |f(\tilde{\lambda}_n)|^2e^{-2\varphi(\tilde{\lambda_n})}
+\sum_n (1+|{\lambda}_n|^2)d_{\rho}(\lambda_n,\tilde{\lambda}_n)
 \max_{z\in D_{\rho}(\lambda_n,\delta)}|f(z)|^2e^{-2\varphi(z)},
\end{eqnarray*}
where we have used Lemma \ref{lem1} (constants only depend on $\delta$). 
Let $\delta_0$ such that the disks $D(\lambda_n,\delta_0|\lambda_n|)$ are disjoint.
Then, when $\delta>0$ is sufficiently small and $z \in D(\lambda_n, \delta|\lambda_n|)$, 
every disk $D(z,\delta_0 |z|/2)$ is 
contained in $D(\lambda_n,\delta_0|\lambda_n|)$.
Now, by Lemma \ref{lem1}, 
for every $z\in D(\lambda_n,\delta|\lambda_n|)$,
\[
 |f(z)|^2e^{-2\varphi(z)}\lesssim \frac{1}{\delta_0^2|z|^2}\int_{D(z,\delta_0|z|/2)}
 |f(w)|^2e^{-2\varphi(w)}dm(w)\le \frac{1}{\delta_0^2|z|^2}\int_{D(\lambda_n,\delta_0 |\lambda_n|)}
 |f(w)|^2e^{-2\varphi(w)}dm(w).
\]
Since $d_{\rho}(z,\lambda_n)\le 2\delta$, we have $1+|\lambda_n|^2\asymp
1+|z|^2$, and hence
\begin{eqnarray*}
 \|f\|_{2,\varphi}^2
\lesssim \sum_n(1+|\tilde{\lambda}_n|^2)
 |f(\tilde{\lambda}_n)|^2e^{-2\varphi(\tilde{\lambda})}
+\delta \sum_n \int_{D_{\rho}(\lambda_n,2\delta)}
 |f(w)|^2e^{-2\varphi(w)}dm(w),
\end{eqnarray*}
where the constants only depend on $\delta$.
It remains to choose $\delta$ sufficiently small.
\end{proof}

Recall that $\Gamma=\Gamma_\alpha=\{e^{\frac{n+1}{2\alpha}}
e^{i\theta_n}\}_{n\ge 0}$, $\theta_n\in \RR$,
is a complete interpolating sequence for $\ahd$. We will need the following 
simple estimates from \cite{BL}.

\begin{lem}\cite[Lemma 2.6]{BL}.
\label{FonctionBL} 
Let $\varphi(r)=\alpha (\log^+ r)^2$ and $\Gamma=\Gamma_{\alpha}$.
The product 
$$
 G(z)=\prod_{\gamma\in \Gamma}\Big(1-\frac{z}{\gamma}\Big)
$$
converges uniformly on compact sets in $\CC$ and satisfies 
$$
 |G(z)|\asymp e^{\varphi(z)}\frac{\dist(z,\Gamma)}{1+|z|^{3/2}}, \qquad z\in \CC,
$$
where the constants are independent of the choice of $\theta_n$. 
Here $\dist(z,\Gamma)$ denotes the Euclidean distance between $z$ and $\Gamma$.
Also
\[
 | G'(\gamma)|\asymp \frac{e^{\varphi(\gamma)}}{1+|\gamma|^{3/2}},\qquad
 \gamma\in\Gamma.
\]
\end{lem}


\begin{lem}
\label{lemma3} 
Let $\Lambda\subset \CC$. Then 
\begin{equation}
\label{unionsepare}
\|f\|_{2,\varphi,\Lambda}\leq c(\Lambda)\|f\|_{2,\varphi},\qquad f\in \ahd,
\end{equation}
if and only if $\Lambda$ is  a finite union of $d_\rho$-separated subsets.
\end{lem}

\begin{proof} If $\Lambda$ is  a finite union of $d_\rho$-separated subsets, then \eqref{unionsepare} follows from 
Lemma \ref{lem1} (iii).  In the opposite direction, let $\Gamma$ and $G$ be as
in Lemma \ref{FonctionBL}, and put   
$$
G_\gamma(z)=\frac{G(z)}{(z-\gamma)G'(\gamma)} \frac{e^{\varphi(\gamma)}}{\gamma}.
$$
Then
$$
 |G_\gamma(z)| \asymp \frac{|\gamma|^{1/2}}{1+|z|^{3/2}}e^{\varphi(z)}
 \frac{\dist(z,\Gamma)}{|z-\gamma|}.
 $$
The function $G_\gamma$ belongs to $\ahd$ and $\sup_{\gamma}\|G_\gamma\|_{2,\varphi}\lesssim 1$ (see \cite[Proof of Theorem 2.5]{BL}).
Hence, by Lemma \ref{FonctionBL}, we have 
\begin{eqnarray*}
1\gtrsim\|G_\gamma\|_{2,\varphi}&\gtrsim&\|G_\gamma\|_{2,\varphi,\Lambda}\geq 
 \sum_{\lambda\in\Lambda\cap D_\rho(\gamma,\beta)}|G_\gamma(\lambda)|^2e^{-2\varphi(\lambda)}(1+|\lambda|^2)\\
&\geq & c \Card(\Lambda\cap D_\rho(\gamma,\beta))
\end{eqnarray*}
(note that $\dist(\lambda,\Gamma)\asymp |\lambda-\gamma|$ since
$\Gamma$ is separated and we can choose $\beta$ such that $D_{\rho}(\gamma,\beta)$
stays sufficiently far from $\Gamma\setminus\{\gamma\}$).
So 
$$
 \sup_{\gamma\in \Gamma}\Card(\Lambda\cap D_\rho(\gamma,\beta))<\infty
$$
for arbitrary  real numbers $\theta_n$ (we can pick $\beta$ such that that 
$D_{\rho}(\gamma,\beta)$ covers $\CC$ when $\gamma$ runs through $\Gamma$ and
$\theta_n$ through $\RR$).  
Hence $\Lambda$ is a finite union of $d_\rho$-separated sequence.
\end{proof}

 \begin{lem}
\label{passageinfty2} 
If $\Lambda$ is $d_\rho$-separated and sampling for  $\mathcal{F}_{(1+\varepsilon)\varphi}^{\infty}$
for some $\varepsilon>0$, then $\Lambda$ is sampling for $ \mathcal{F}_{\varphi}^{2}$.
 \end{lem}

\begin{proof} We are going to use the same Beurling duality argument as in 
\cite[Theorem 36]{MMO}
(see also \cite[pp. 348-358]{B}, \cite[pp.36-37]{S3}). Let 
$$
 \mathcal{F}_{\varphi}^{\infty,0}=\big\{f\in \mathcal{F}_{\varphi}^{\infty} \text{ : } 
 \lim_{|z|\to \infty}|f(z)|e^{-(1+\varepsilon)\varphi(z)}=0\big\}.
$$
By the sampling property,
the operator family $T_z:\{f(\lambda)\}_{\lambda\in\Lambda}\longmapsto
e^{-(1+\eps)\varphi(z)}f(z)$, $z\in\CC$, is uniformly bounded from $\{f|_\Lambda:\, 
f\in \mathcal{F}_{\varphi}^{\infty,0}\}\subset c_0$ to $\CC$.
Hence, by duality, there exists a family  $(g(z,\lambda))_{\lambda\in \Lambda}$ such that  
$$e^{-(1+\varepsilon)\varphi(z)}f(z)=\sum_{\lambda\in \Lambda}e^{-(1+\varepsilon)\varphi(\lambda)}f(\lambda)g(z,\lambda), \qquad f\in \mathcal{F}_{(1+\varepsilon)\varphi}^{\infty,0}.$$
and  $\sup_z\sum_\lambda|g(z,\lambda)|<\infty$. 
Let now  $\Gamma=\Gamma_{{\varepsilon}\alpha}=\{ e^{\frac{n+1}{2\alpha{\varepsilon}}}e^{i\theta_n}\}$, 
and consider the function $G\in  \mathcal{F}_{{\varepsilon}\varphi}^\infty$ of  Lemma \ref{FonctionBL}  vanishing on $\Gamma$. 
When 
$e^{\frac{n+1}{2\alpha{\varepsilon}}}\leq |z|\leq  e^{\frac{n+2}{2\alpha{\varepsilon}}}$,
let $\gamma_z= e^{\frac{n+2}{2\alpha{\varepsilon}}}e^{i\theta_n}$, so that
\[
  e^{-1/(2\alpha\varepsilon)}\le |{z}/{\gamma_z}|\le 1. 
\]
Set 
$$
P_z(w)=\frac{G(w)}{(w-\gamma_z)G'(\gamma_z)}\frac{w^2}{z^2}.
$$
For $w\in\CC$ we have 
\begin{equation}\label{estimate-peak}
|P_z(w)|\asymp e^{{\varepsilon}(\varphi(w)-\varphi(\gamma_z))}\frac{|w|^{1/2}}{|z|^{1/2}} \frac{\dist(w,\Gamma)}{|w-\gamma_z|}
\lesssim e^{{\varepsilon}(\varphi(w)-\varphi(z))}\frac{|w|^{1/2}}{|z|^{1/2}} 
\frac{\dist(w,\Gamma)}{|w-\gamma_z|}.
\end{equation}
Given  $f\in \ahd$,  by Lemma \ref{lem1} (iii) and  \eqref{estimate-peak}  we have 
$w\mapsto f(w)P_z(w)\in  
\mathcal{F}_{(1+{\varepsilon})\varphi}^{\infty,0}$ and hence
$$
e^{-(1+\varepsilon)\varphi(z)} f(z) P_z(z)=\sum_{\lambda\in \Lambda}e^{-(1+\varepsilon)\varphi(\lambda)} f(\lambda)P_z(\lambda)g(z,\lambda).
$$
Since  $|P_z(z)|\asymp 1$, again by  \eqref{estimate-peak} we obtain that 
$$
|f(z)|e^{-\varphi(z)}
\lesssim 
\sum_{\lambda\in \Lambda} |f(\lambda)| e^{-\varphi(\lambda)}\frac{|\lambda|^{1/2}}{|z|^{1/2}} \frac{\dist(\lambda,\Gamma)}{|\lambda-\gamma_z|}|g(z,\lambda)|.
$$
Since  $\sum_\lambda|g(z,\lambda)|<\infty$, H\"older's inequality 
and \eqref{kernel-estimate} give us that 
\begin{eqnarray}\label{eqsamplingf}
 |f(z)|^2e^{-2\varphi(z)}&\lesssim& 
 \bigg( \sum_{\lambda\in \Lambda}|f(\lambda)|^2e^{-2\varphi(\lambda)} 
 \frac{|\lambda|}{|z|} \frac{\dist(\lambda,\Gamma)^2}{|\lambda-\gamma_z|^2}\bigg)
 \bigg(\sum_{\lambda\in \Lambda}|g(z,\lambda)|^2\bigg) \nonumber\\
 &\lesssim & 
 \bigg(\sum_{\lambda\in \Lambda}\frac{|f(\lambda)|^2}{\bk_\lambda(\lambda)} 
 \frac{1}{|\lambda z|} \frac{\dist(\lambda,\Gamma)^2}{|\lambda-\gamma_z|^2} \bigg)
 \bigg( \sum_{\lambda\in \Lambda}|g(z,\lambda)|\bigg)^2 \nonumber\\ 
 &\lesssim & 
 \sum_{\lambda\in \Lambda}\frac{|f(\lambda)|^2}{\bk_\lambda(\lambda)} 
 \frac{1}{|\lambda z|} \frac{\dist(\lambda,\Gamma)^2}{|\lambda-\gamma_z|^2}. 
\end{eqnarray}
It remains to verify that
\begin{equation}\label{integrallambda}
I(\lambda)=\frac{1}{|\lambda|}\int_\CC
\frac{\dist(\lambda,\Gamma)^2}{|\lambda-\gamma_z|^2|z|}dm(z)<\infty
\end{equation}
uniformly in $\lambda$, since by \eqref{eqsamplingf}, we then 
obtain the sampling inequality 
$$
\|f\|_{2,\varphi}^{2}\lesssim 
\sup_\lambda I(\lambda)\|f\|_{2,\varphi,\Lambda}^{2}.  
$$

We will now show \eqref{integrallambda}. 
Since $\dist (\lambda,\Gamma)\le |\lambda-\gamma_z|$ and $|\gamma_z| \asymp|z|$, we have 
$$
 \frac{1}{|\lambda|}\int_{z:\; |\gamma_z|<2|\lambda|}
 \frac{\dist(\lambda,\Gamma)^2}{|\lambda-\gamma_z|^2|z|}dm(z)
 \lesssim\frac{1}{|\lambda|}\int_{|z|\lesssim |\lambda|}\frac{1}{|z|}dm(z)\lesssim 1.
$$
If $|\gamma_z|\geq 2|\lambda|$, then $|\lambda-\gamma_z|\geq |\gamma_z|/2\asymp |z|$ 
and $\dist(\lambda,\Gamma)\lesssim|\lambda|$, so that 
$$
 \int_{z:\; 2|\lambda|\leq  |\gamma_z|}\frac{1}{|\lambda||z|} 
 \frac{\dist(\lambda,\Gamma)^2}{|\lambda-\gamma_z|^2}dm(z)
 \lesssim|\lambda|\int_{|\lambda|\lesssim |z|}\frac{1}{|z|^3}dm(z)\lesssim1.
$$
Hence \eqref{integrallambda} is established and the proof is complete.
\end{proof}

\subsection{De Branges spaces}

In order to investigate the Hilbertian counterpart of the above result we will identify
the Fock space with a de Branges space.
Let $G$ be the generating function associated with the sequence 
$\Gamma$ defined by \eqref{bls} with $\theta_n=-\pi/2$. 
Recall that the de Branges space associated with $G$ (see \cite{br}) is given by
\[ 
 \mathcal{H}(G):=\{f \text{ entire}:f/G\in H^2(\CC^+)\text{ and }f^*/G\in H^2(\CC^+)\},
\]
where $f^*(z)=\overline{f(\bar z)}$ and $H^2(\mathbb{C}^+)$ 
stands for the standard Hardy space.  The space $\mathcal{H}(G)$ is normed by 
\[
 \|f\|_{\mathcal{H}(G)}^2:=\int_{\RR}\left|\frac{f(x)}{G(x)}\right|^2dx,
 \quad f\in \mathcal{H}(G).
\]
We have the following result.

\begin{prop}
\label{asdb}
The space $\ahd$ is a de Branges space: $\ahd=\mathcal{H}(G)$.
\end{prop}

\begin{proof}
We know that the normalized reproducing kernels   
$\{\K_\gamma\}_{\gamma\in \Gamma}=\{\bk_\gamma/\|\bk_\gamma\|_{\varphi,2}\}_{\gamma \in \Gamma}$
form a Riesz basis in $\mathcal F_\varphi^2$. Then the biorthogonal family
$$
\frac{\|\bk_\gamma\|_{\varphi,2}}{G'(\gamma)}
\cdot \frac{G(z)}{z-\gamma}\ , \qquad \gamma \in \Gamma,
$$
is a Riesz basis in $\mathcal F_\varphi^2$. By formula 
\eqref{kernel-estimate} and the 
estimate for $|G'(\gamma)|$ in Lemma \ref{FonctionBL} we conclude that the above 
biorthogonal system is of the form $a_\gamma |\gamma|^{1/2}\cdot 
{G(z)}/(z-\gamma)$, where $|a_\gamma| \asymp 1$. Hence, any function $f$
in the space $\mathcal F_\varphi^2$ can be written as
\begin{equation}
\label{1}
f(z)  = \sum_{\gamma \in \Gamma} c_\gamma |\gamma|^{1/2}\cdot 
\frac{G(z)}{z-\gamma}, 
\end{equation}
where $\{c_\gamma\} \in \ell^2$ and  
$\|f\|_{\varphi,2} \asymp \|\{c_\gamma\}\|_{\ell^2}$.
Writing for simplicity $\gamma_n = -iy_n$, we have
$$
\frac{f(z)}{G(z)} = \sum_n \frac{c_n y_n^{1/2}}{z+iy_n},
$$
and the series converges in the Hardy space 
$H^2= H^2(\mathbb{C}^+)$, since $\gamma$ satisfies the Carleson condition and so 
is an $H^2$-interpolating sequence.
Analogously, if we put $\Theta = G^*/G$, we get
$$
\frac{f^*(z)}{G(z)} = \sum_n \overline c_n y_n^{1/2}\frac{\Theta(z)}{z-iy_n},
$$
again the series converges in $H^2$, since $\Theta$ is an interpolating 
Blaschke product (with zeros $i y_n$). We conclude that $f/G$ and $f^*/G$ are in $H^2$.
Conversely, any function in $\mathcal{H}(G)$ can be written as a series of the form \eqref{1},
since the functions $({\rm Im}\,\gamma)^{1/2} G(z)/(z-\gamma)$ form a Riesz basis in $\mathcal{H}(G)$
whenever the zero set of $G$ is an interpolating sequence.
\end{proof}

It is also possible to have the comparison 
with the integral over the positive
or negative rays.
%
\begin{cor}
Let $G$ be the generating function of the sequence \eqref{bls} 
with $\theta_n=-\pi/2$. Then the measure $dx/|G(x)|^2$ is sampling on 
$\RR_+$, $\RR_-$ or $\RR$ for $\ahd$: for every $f\in \ahd$,
\[
 \|f\|_{\varphi,2}^2\asymp \int_{\RR_+}\left|\frac{f(x)}{G(x)}\right|^2dx
 \asymp \int_{\RR_-}\left|\frac{f(x)}{G(x)}\right|^2dx
 \asymp \int_{\RR}\left|\frac{f(x)}{G(x)}\right|^2dx.
\]
\end{cor}

This answers a question appearing in equation (4.3.1) in \cite{M}.

\begin{proof}
Since $\|f\|_{\varphi,2}\asymp\int_{\RR}\left|\frac{\dst f(x)}{\dst G(x)}\right|^2dx$, it is sufficient to prove that
\[
 \int_{\RR}\left|\frac{f(x)}{G(x)}\right|^2dx
 \lesssim \int_{\RR_+}\left|\frac{f(x)}{G(x)}\right|^2dx.
\]
The sequence $\Gamma$ given by \eqref{bls} (with $\theta_n=0$) is sampling, and hence,
according to Corollary \ref{pertp=2}
there exists $\delta>0$ and $C$ depending only on $\delta$
such that for every perturbation $\widetilde{\Gamma}$ 
of $\Gamma$ with $d_{\rho}(\gamma_n,\tilde{\gamma}_n)\le\delta$
we have
\[
 \frac{1}{C}\|f\|_{\varphi,2}\le \|f\|_{\varphi,2,\widetilde{\Gamma}}.
\]

Consider now the sequence $\Lambda=\{\lambda_n\}_{n\ge 0}$ 
such that $\lambda_n>\gamma_n$  and  $d_{\rho}(\gamma_n, \lambda_n) = \delta$. 
We set $I_n=[\gamma_n, \lambda_n)$, $n\ge 0$. Clearly, the intervals $I_n$
are disjoint if $\delta$ is sufficiently small. 
By the mean value theorem,  there exists $x_n\in I_n$ such that
\[
 \int_{I_n} \left|\frac{f(x)}{G(x)}\right|^2dx = |I_n|\times
 \left|\frac{f(x_n)}{G(x_n)}\right|^2.
\]
It is also clear that $|I_n|\asymp 1+x_n$. Since $x_n\in\RR^+$, using the estimate
in Lemma \ref{FonctionBL}
and taking into account that $\dist(x_n,-i\Gamma)\asymp 1+x_n$, we 
get
\begin{eqnarray*}
 \int_{\RR_+} \left|\frac{f(x)}{G(x)}\right|^2dx
 &\geq& \sum_{n\ge 0}\int_{I_n} \left|\frac{f(x)}{G(x)}\right|^2dx
 \asymp \sum_{n\ge 0}(1+x_n)\left|\frac{f(x_n)}{G(x_n)}\right|^2\\
 &\asymp& \sum_{n\ge 0}|f(x_n)|^2(1+x_n^2)e^{-2\varphi(x_n)}
 \asymp \sum_{n\ge 0}\frac{|f(x_n)|^2}{\bk_{x_{n}}(x_{n})}.
\end{eqnarray*}
Since $d_{\rho}(x_{n},\gamma_n)\le \delta$, the sequence $\widetilde{\Gamma}
:=\{x_{n}\}_{n\ge 0}$ is sampling for $\ahd$, and we get
\[
 \int_{\RR_+}\left|\frac{f(x)}{G(x)}\right|^2 dx
 \gtrsim \|f\|_{\varphi,2,\widetilde{\Gamma}}\ge\frac{1}{C}\|f\|_{\varphi,2}.
\]
\end{proof}

The following result shows that we have a similar situation in $\ahc$.

\begin{prop}\label{samplingaxe}
Every half-line starting from the origin is sampling for $\ahc$.
\end{prop}
\begin{proof}
Pick $f\in \ahc$ 
with $\|f\|_\varphi =1$. Define
\[
 F(z)=f(z)e^{-\alpha(\Log z)^2},
\]
cutting  the plane at the positive real axis. Then $F$ is an analytic function
in $\C\setminus \RR^+_*$. Moreover,
\begin{eqnarray}\label{Hi-estim}
 |F(z)|=|f(z)|e^{-\alpha (\log^2|z|-(\arg z)^2)}\asymp |f(z)|e^{-\alpha (\log |z|)^2}.
\end{eqnarray}
Hence $F\in H^{\infty}(\C\setminus \RR^+_*)$ implying that 
$$
 \sup_{z\in\C}|f(z)|e^{-\alpha(\log_+|z|)^2} 
 \asymp \sup_{z\in \C\setminus \RR^+_*} |F(z)|
 =\sup_{z\in \RR_+}|F(z)|
  =\|f\|_{\varphi,\infty,\RR_+},
$$
which proves the claim.
\end{proof}


\subsection{Density results}\label{section3}

\begin{lem}
If $\Lambda$ is $d_\rho$-separated then $D^-({\Lambda})\le D^{+}(\Lambda)<\infty$.
\end{lem}


\begin{proof}
As already mentioned in the beginning of Section \ref{section2}, 
when $\Lambda$ is $d_\rho$-separated, 
then there exists $c$ such that the Euclidean disks $D(\lambda,
c|\lambda|)$, $\lambda\in\Lambda$, are disjoint. 
A standard argument, based for instance on the consideration of the Euclidean
area of $\cA(x,R x)$ and that of the disks $D(\lambda,c|\lambda|)$, $\lambda\in
\Lambda\cap\cA(x,R x)$, shows that 
this implies in particular that 
for fixed
$\eta>0$
every annulus $\cA(x,\eta x)$ contains a uniformly bounded number
of points of $\Lambda$ (this number depends on $\eta$):
\[
 \Card(\Lambda\cap \cA(x,\eta x))\le M,\quad x>0.
\]
Suppose now that $R>R_0$ and $r>r_0$ with $R_0,r_0$ big enough. 
Let $N$ be the least integer such that $r\eta^N\ge rR$ so that
$N\asymp \log R/\log\eta$. Then
$$
\Card(\Lambda\cap \cA(r,Rr))=\sum_{\lambda\in\Lambda\cap \cA(r,Rr)} 1\le\sum_{n=1}^N\Card(\Lambda\cap\cA(r\eta^{n-1},r\eta^n))\lesssim \frac{M}{\log\eta}\log R.
$$
\end{proof}

\begin{prop}\label{conddiskdensity}
If $\Lambda$ is a zero sequence for $\ahc$ then
\begin{eqnarray}\label{diskdensity}
\liminf_{R\to +\infty}\frac{\Card (\Lambda\cap D(0,R))}{\log R} \le 2\alpha.
\end{eqnarray}
\end{prop}

\begin{proof}
Suppose there is a function $g$ that vanishes on $\Lambda$ with
 \begin{equation}\label{densitelambda}
 \liminf_{R\to +\infty}\frac{\Card (\Lambda\cap D(0,R))}{\log R} >2\alpha.
 \end{equation}
Assuming $g(0)\neq 0$ (otherwise divide by a suitable power of $z$ 
which does not change the other zeros of $g$ and gives a function still in 
$\ahc$), Jensen's formula yields for every
$R>0$, 
$$
\sum_{\lambda\in{\Lambda} : |\lambda|<R}\log\frac{R}{|\lambda|} = 
\frac{1}{2\pi}\int_0^{2\pi} \log|g(Re^{i\theta})|\,d\theta -\log|g(0)| 
  \leq \alpha (\log R)^2 +C.
$$
Denote now by $n_g(R)$ the number of zeros of $g$ in $D(0,R)$. Then
$$
\sum_{\lambda\in{\Lambda} : |\lambda|<R}\log\frac{R}{|\lambda|} = 
\int_0^R\frac{n_g(t)}{t}\,dt.
$$
From \eqref{densitelambda} we deduce that
for $\varepsilon>0$ small enough there exists $R_0>0$ such that
for every $R\geq R_0$, 
$$n_g(R)=\Card(\Lambda \cap D(0,R))\geq 2\alpha(1+\varepsilon)\log R.$$
Then for every $R\geq R_0$,
$$\sum_{\lambda\in{\Lambda} : |\lambda|<R}\log\frac{R}{|\lambda|} \geq
\int_{R_0}^R \frac{n_g(t)}{t}dt
 \geq 2\alpha(1+\varepsilon)\int_{R_0}^R \frac{\log t}{t}dt
 \geq 2\alpha (1+\varepsilon)\left(\frac{(\log R)^2}{2}-\frac{(\log R_0)^2}{2}\right).
 $$
It follows that
$$
\alpha (1+\varepsilon)\Big((\log R)^2-(\log R_0)^2\Big)\leq \alpha (\log R)^2+C$$
which is impossible when $R$ is big.
\end{proof}

We may deduce the following corollary.

\begin{cor} \label{densitezero}
If $\Lambda$ satisfies \eqref{diskdensity}
then $D^{-}(\Lambda)\leq 2\alpha.$
\end{cor}


\begin{proof}
By contraposition, suppose that $D^-(\Lambda)>2\alpha$. Then there
are $R_0$ and $r_0$ such that for every $R>R_0$ and $r>r_0$ we have
\[
 \frac{\Card(\Lambda\cap\cA(r,Rr))}{\log R}\ge 2(1+\eps)\alpha,
\]
for a suitable fixed $\eps$.
Set $\eta=\max(R_0,r_0)$ and let $x=\eta^{N+\kappa}\in[\eta^N,\eta^{N+1})$
($\kappa\in [0,1)$) be big, then
\begin{eqnarray*}
\frac{\Card(\Lambda\cap D(0,x))}{\log  x} 
 &\ge&\frac{\sum_{k=1}^{N-1}\Card(\Lambda\cap\cA(\eta^k,\eta^{k+1}))}
 {\log x}\\
 &\ge&\frac{2(1+\eps)\alpha (N-1)\log \eta}{(N+\kappa)\log \eta}
 \longrightarrow 2(1+\eps)\alpha, \quad N\to\infty,
\end{eqnarray*}
i.e., $\Lambda$ does not satisfy \eqref{diskdensity}.
\end{proof}

The two preceding results together team up in:

\begin{cor}\label{densitezero1}
If $\Lambda$ is a zero sequence for $\ahc$,  
then $D^{-}(\Lambda)\leq 2\alpha.$
\end{cor}


\section{Proof of the result on Riesz bases} 
\label{section7}

``$\Longleftarrow$'': We use Bari's Theorem \cite[p. 132]{N}. As in the
introduction,
let  $\bk_\lambda$ be the reproducing kernel of  $\ahd$ and let 
$\K_\lambda=\bk_\lambda/\|\bk_\lambda\|_{\varphi,2}$ be the normalized kernel at 
$\lambda$. Let  $F$ be an entire function with simple zeros at each $\lambda\in\Lambda$, 
$$ 
 F(z):= \prod_{n\geq 0} \big(1-\frac{z}{\lambda_n}\big),\qquad z\in \CC,
$$
and set 
$$g_\lambda(z)=\frac{F(z)}{F'(\lambda)(z-\lambda)}\|\bk_\lambda\|_{\varphi,2}, \qquad z\in \C.$$
If the functions $g_{\lambda}$ are in $\ahd$, then the family
$\{g_\lambda\}_{\lambda\in \Lambda}$ is biorthogonal to 
$\cK_\Lambda:=\{\K_\lambda\}_{\lambda\in \Lambda}$. 
Hence to show that  $\cK_\Lambda$ is Riesz basis it suffices to prove 

 \begin{enumerate}
 \item[(i)] $F/(\cdot -\lambda)\in \ahd$ for   $\lambda\in \Lambda$; 
 \item[(ii)] $\cK_\Lambda$ is complete : $\bigvee\{\K_\lambda,\; \; \lambda\in \Lambda\}=\ahd$;
 \item[(iii)] $\displaystyle \sum_{\lambda\in \Lambda}\big|\langle f, \K_\lambda\rangle\big|^2\lesssim \|f\|_{\varphi,2}^{2}$;
 \item[(iv)] $\displaystyle \sum_{\lambda\in \Lambda}\big|\langle f, g_\lambda\rangle \big|^2\lesssim \|f\|_{\varphi,2}^{2}$.
\end{enumerate}



To prove (i),  let 
$|z|=e^t$ with $|\lambda_{n-1}|\leq |z|\leq| \lambda_{n}|$ and suppose that 
$\dist (z,\Lambda)=|z-\lambda_{n-1}|$. Let $m\in\NN$ be such that  
$$
\frac{m}{2\alpha}-\frac{1}{4\alpha}\le t< \frac{m}{2\alpha}+\frac{1}{4\alpha}.
$$
Then, since $|\lambda_n|=\gamma_n e^{\delta_n}\asymp \gamma_m$ by the above,
and $\delta_m$ is uniformly bounded,
we have that $|m-n|$ is uniformly bounded in $|z|$, and $\log |\lambda_s|-t$ is bounded uniformly in $z$ and $s$ between $m$ and $n$. 
We use that for $d_{\rho}$-separated sequences the behavior of the
function $F$ is essentially given by the first $n$ terms.   We have 
\begin{eqnarray}\label{estimLogF}
 \log |F(z)|&=&\sum_{0\le k\le n-2} \log\frac{|z|}{|\lambda_k|}+\log|1-\frac{z}{\lambda_{n-1}}|+O(1)\\
 &=&\sum_{0\le k\le n-1}
 \bigl(t-\frac{k+1}{2\alpha}\bigr)+\log\dist(z,\Lambda)-t -\sum_{0\le k\leq n-1}\delta_k+O(1)
 \nonumber\\
 &=&\sum_{0\le k\le m-1}
 \bigl(t-\frac{k+1}{2\alpha}\bigr)+\log\dist(z,\Lambda)-t -\sum_{0\le k\leq m-1}\delta_k+O(1)
 \nonumber\\
 &=&mt-\frac{m(m+1)}{4\alpha}-\sum_{0\leq k\leq m-1}\delta_k+\log\dist(z,\Lambda)-t+O(1) 
 \nonumber\\
 &=&\alpha t^2- \frac{3}{2}t +\log \dist(z,\Lambda)-\sum_{0\leq k\leq m-1}\delta_k+O(1),\qquad t\to\infty.\nonumber
\end{eqnarray}
Next, if $m=lN+r$, $0\le r<N$, then 
\begin{eqnarray}\label{sumdelta}
\Bigl|\sum_{k=0}^{m-1}\delta_k\Bigr|\le\sum_{j=0}^{l-1}\Bigl|\sum_{i=0}^{N-1}
\delta_{jN+i}\Bigr|+\Bigl|\sum_{i=0}^{r-1}
\delta_{lN+i}\Bigr|\le {lN}\delta+O(1)=2\alpha\delta t+O(1).
\end{eqnarray}
Therefore, for some $\eta>0$, 
\begin{equation}
\label{psi-estim}
 e^{\varphi(z)} \frac{\dist(z,\Lambda)}{(1+|z|)^{2-\eta}}\lesssim |F(z)|\lesssim e^{\varphi(z)} \frac{\dist(z,\Lambda)}{(1+|z|)^{1+\eta}}, 
 \qquad z\in\C.
\end{equation}

This proves (i).

Next we pass to property (iii). By assumption, $\Lambda$ is $d_\rho$-separated, and so by Lemma \ref{lem1} we have 
$$
 \sum_{\lambda\in \Lambda}\big|\langle f, \K_\lambda\rangle\big|^2= 
 \sum_{\lambda\in \Lambda}\frac{|f(\lambda)|^2}{\|\bk_\lambda\|^{2}_{\varphi,2}}
 \lesssim\|f\|_{\varphi,2}^{2}.
$$

Let us turn to  (ii). By Lemma \ref{FonctionBL} we have 
$$
 |G(z)|\asymp e^{\varphi(z)}\frac{\dist(z,\Lambda)}{1+|z|^{3/2}}, \qquad z\in\C.
$$
If $f\in\ahd$, then by Lemma \ref{lem1} (iii), 
$|f(z)|=o( e^{\varphi(z)}/(1+|z|))$ and so 
\begin{eqnarray}\label{estimfG}
 |f(z)/G(z)|=o\bigg( \frac{1+\sqrt{|z|}}{\dist(z,\Gamma)}\bigg), \qquad z\in\C.
\end{eqnarray}
Let $|z|=e^t$ and let $n$ be such that
$$
 \frac{n}{2\alpha}-\frac{1}{4\alpha}\leq t<\frac{n}{2\alpha}+\frac{1}{4\alpha}.
$$
If we denote by $k(n)$ the integer such that the point $\lambda_{k(n)}
\in \Lambda$ is the closest  to $z$, then, by condition (b),
$|k(n) - n|$ is uniformly bounded. Hence, keeping in mind that 
$\Lambda$ is separated and $|\lambda_m| \asymp |z|$, $|m-n| \le |k(n)-n|$, we have
\begin{eqnarray}
 \label{estimF}
 |F(z)|  \asymp \prod_{k=1}^{k(n)} 
 \bigg|1-\frac{z}{\lambda_k}\bigg| 
 \asymp \frac{\dist(z,\Lambda)}{|z|} 
 \prod_{k=1}^{n-1}\frac{|z|}{|\lambda_k|}.
\end{eqnarray}

Thus, 
$$
 \frac{|F(z)|}{|G(z)|}\asymp \frac{\prod_{k=0}^{n-1} 
 |z/\lambda_k|}{\prod_{k=0}^{n-1} |z/\gamma_k|}
 \cdot \frac{\dist(z,\Lambda)}{\dist(z,\Gamma)}\asymp
 \exp\left(-\sum_{k=0}^{n-1}\delta_k\right) \frac{\dist(z,\Lambda)}{\dist(z,\Gamma)}.
$$
As in \eqref{sumdelta}, and recalling that $|m-n|$ is uniformly bounded, we have
$|\sum_{k=0}^{n-1}\delta_k|\le 2\alpha\delta t+O(1)$, where $2\alpha\delta<1/2$.
Therefore, 
$$
 \frac{1}{|z|^{2\alpha\delta}}\frac{\dist(z,\Lambda)}{\dist(z,\Gamma)}
 \lesssim \frac{|F(z)|}{|G(z)|},\qquad z\notin\Gamma.
$$
If $z$ is $d_{\rho}$-far from $\Lambda$, we have 
$\dist(z,\Lambda)/|z|^{2\alpha\delta}\asymp
|z|^{1-2\alpha\delta}>>\sqrt{|z|}$ so that with \eqref{estimfG} in mind we see that 
$F\notin \ahd$.  Now, if $\cK_\Lambda$ is not complete, then there exists 
a nonzero $f\in \ahd$ vanishing on $\Lambda$. So  $f=FS$ for some entire 
function $S$, and
$$
 |S(z)|=\left|\frac{f(z)}{F(z)}\right|=\left|\frac{f(z)}{G(z)}\frac{G(z)}{F(z)}\right|
 =o\bigg( \frac{1+|z|^{1/2+2\alpha\delta}}{\dist(z,\Lambda)}\bigg).
$$
Now, for every $R>0$, we can find a closed contour $C_R$ surrounding $D(0,R)$ which is 
$d_{\rho}$-separated from $\Lambda$ (and not meeting $\Gamma$). 
We have $\max_{z\in C_R} |S(z)|=o(1)$, $R\to \infty$, 
so that $\sup_{|z|\le R}|S(z)|=o(R)$, and hence 
$S$ vanishes identically, which is impossible.  
Statement (ii) is proved. 

It remains to show (iv). 
Let us estimate $F(\gamma_m)$ and $F'(\lambda_n)$. 
As above, we only need to consider the terms $k \leq n-1$.
By \eqref{estimF},
\begin{equation}
\label{bab1}
 |F(\gamma_m)|\asymp \frac{\dist(\gamma_m,\Lambda)}{\gamma_m}
 \prod_{0\leq k\leq m-1}\frac{\gamma_m}{|\lambda_k|}
\end{equation}
and 
\begin{equation}
\label{bab2}
 |F'(\lambda_n)|\asymp \frac{1}{|\lambda_n|}
 \prod_{0\leq k\leq n-1}\left|\frac{\lambda_n}{\lambda_k} \right| .
\end{equation}
Since the family $\{\K_\gamma\}_{\gamma\in \Gamma}$ is a Riesz basis,
we can  write
$$
 f=\sum_{m\geq 0}a_m\K_{\gamma_m},\qquad  (a_m)_{m \ge 0} \in \ell^2,
$$
and the sum in  (iv) becomes
$$
 \sum_{n\geq 0}\bigg|\sum_{m\geq 0} a_m 
 \underbrace{ \frac{F(\gamma_m)}{F'(\lambda_n)(\gamma_m-\lambda_n)}
 \cdot\frac{\|\bk_{\lambda_n}\|_{\varphi,2}}{\|\bk_{\gamma_m}\|_{\varphi,2}}}_{A_{n,m}}\bigg|^2.
$$
It remains to check that the matrix $[A_{n,m}]$ 
defines a bounded operator in $\ell^2$.

It follows from \eqref{bab1}, \eqref{bab2} and \eqref{kernel-estimate} that
$$
\begin{aligned}
|A_{n,m}| & \asymp \frac{\dist(\gamma_m,\Lambda)}{|\gamma_m - \lambda_n|}
\cdot \frac{\gamma_{m}^{m}}{|\lambda_{n}|^{n}} \cdot
\left(\prod_{0 \leq k\leq n-1} {|\lambda_k|}\right)\cdot
\left(\prod_{0 \leq k\leq m-1} {|\lambda_k|}\right)^{-1} \cdot
e^{ \alpha(\log |\lambda_n|)^2 - \alpha(\log\gamma_m)^2}  \\
& =\frac{\dist(\gamma_m,\Lambda)}{|\gamma_m - \lambda_n|} e^{c(n,m)},
\end{aligned}
$$
where 
$$
\begin{aligned}
c(n,m) & =  \frac{m(m+1)}{2\alpha}-\Big(\frac{n+1}{2\alpha}+\delta_n\Big) n -
\sum_{k=0}^{m-1}\Big(\frac{k+1}{2\alpha}+\delta_k\Big) \\
& \qquad \qquad - \sum_{k=0}^{n-1}\Big(\frac{k+1}{2\alpha}+\delta_k\Big)
 +\alpha\Big( \frac{n+1}{2\alpha}+\delta_n\Big)^2-\frac{(m+1)^2}{4\alpha} \\
 & = -\frac{m-n}{4\alpha} + \sum_{k=0}^{n-1}\delta_k - \sum_{k=0}^{m-1}\delta_k+O(1). 
\end{aligned}
$$

By condition (b) there exists $M$ such that $|\gamma_m - \lambda_n| \asymp |\lambda_n|$
when $n>m+M$, $|\gamma_m - \lambda_n| \asymp |\gamma_m|$ when 
$m>n+N$,  and  $|\gamma_m| \asymp |\lambda_n|$ for $|m-n|\le M$.
\medskip

$\bullet$ Let $|m-n|\le M$. Then it is clear that $|A_{n,n}| \lesssim 1$
with a bound independent of $n$ and $m$.
\medskip 

$\bullet$ If  $m>n+M$, then 
$$
 |A_{n,m}|\asymp \frac{\dist(\gamma_m,\Lambda)}{\gamma_m} 
 \exp\Big(-\frac{m-n}{4\alpha}+\sum_{k=n}^{m-1}\delta_k+O(1)\Big).
$$
\smallskip

$\bullet$ If  $n> m+M$, then $|\gamma_m-\lambda_n|\asymp |\lambda_n|$ and so 
$$
\begin{aligned}
 |A_{n,m}| & \asymp \frac{\dist(\gamma_m,\Lambda)}{\gamma_m} 
 \cdot \frac{\gamma_m}{|\lambda_n|} 
 \cdot \exp\Big(\frac{n-m}{4\alpha}+\sum_{k=m}^{n-1}\delta_k+O(1)\Big) \\
 & \asymp  \frac{\dist(\gamma_m,\Lambda)}{\gamma_m} 
 \cdot \exp\Big(-\frac{n-m}{4\alpha}+\sum_{k=m}^{n-1}\delta_k+O(1)\Big),
\end{aligned}
$$
since $\gamma_m/|\lambda_n| \asymp \exp\big(\frac{m-n}{2\alpha}\big)$.

Thus, 
\begin{equation}
\label{matrix}
|A_{n,m}|\asymp \frac{\dist(\gamma_m,\Lambda)}{\gamma_m} 
\exp\Big(-\frac{|m-n|}{4\alpha}+\sum_{k=m}^{n-1}\delta_k+O(1)\Big),
\end{equation}
where the sum is taken with negative sign if $m>n$.
It follows from (c) that $|A_{n,m}| \lesssim \exp (- \delta|n-m|)$
for some $\delta >0$, and so 
the matrix $[A_{n,m}]$ defines a bounded operator in $\ell^2$.
Statement (iv) is proved. 
\vspace{1em}


``$\Longrightarrow$'':

(a) By Corollary~\ref{lemma2x}, $\Lambda$ is $d_\rho$-separated. 
\smallskip

(b) Suppose that $(\delta_n) \notin\ell^\infty$. 
Then there exists an infinite subsequence of indices $\mathcal{N} = \{n_k\}$
such that for each $k$ there exists $m_k$ such that 
$d_\rho(\lambda_{n_k}, \gamma_{m_k}) \lesssim 1$, but $|n_k-m_k| \to\infty$
as $k\to \infty$.
Passing possibly to another subsequence (also denoted by $\lambda_{n_k}$ to not
overcharge notation) we can suppose that this subsequence is in an angle. 
Now, since $\mathcal{F}_{\varphi}^2$ 
is rotation invariant, $\{e^{i\theta} \lambda_n\}$  
is also a complete interpolating sequence for any $\theta \in \mathbb{R}$
for which the subsequence $(\lambda_{n_k})_k$ is $d_\rho$ separated from $\Gamma$.
Thus, we may assume without loss of generality that
$$
d_\rho(\lambda_{n_k}, \Gamma) \asymp 1,
$$
that is, $|\lambda_{n_k} - \gamma| \ge c|\lambda_{n_k}|$, $\gamma\in \Gamma$  
(with constants independent of $k$) and also that
$$ 
d_\rho(\gamma_{m_k}, \Lambda) \asymp 1.
$$ 
To simplify the notations we write $n$ and $m$ in place of $n_k$ and $m_k$, 
 

Now let us consider 
$$
 A_{n,m} = \langle g_{\lambda_n},
 \K_{\gamma_m}\rangle = \frac{F(\gamma_m)}{F'(\lambda_n)(\gamma_m-\lambda_n)}
 \cdot\frac{\|\bk_{\lambda_n}\|_{\varphi,2}}{\|\bk_{\gamma_m}\|_{\varphi,2}}.
$$
Observe that we have assumed conditions (a) and (c) which ensure that $g_{\lambda}\in\ahd$
(cf.\ proof of (i) in Bari's theorem in the beginning of this section). 
As in \eqref{bab1}--\eqref{bab2} and taking into account 
that $\dist(\gamma_m,\Lambda)\asymp \gamma_m$,
$$                  
 |F(\gamma_m)|  \asymp 
 \frac{\gamma_m^n}{|\lambda_0\lambda_1 \dots\lambda_{n-1}|}.
$$
Analogously, 
$$
|F'(\lambda_n)| \asymp \frac{|\lambda_n|^{n-1}}{|\lambda_0\lambda_1 \dots\lambda_{n-1}|}.
$$
Thus, using the estimate  \eqref{kernel-estimate} for the norm of the reproducing kernel
as well as the estimates $\gamma_m \asymp |\lambda_n| 
\asymp |\gamma_m-\lambda_n|$, we get
$$
|A_{n,m}| \asymp \frac{\gamma_m^n} {|\lambda_n|^n} e^{\alpha(\log^2 |\lambda_n| - 
\log^2 \gamma_m)}.
$$
Note that $ \log\gamma_m = \frac{m+1}{2\alpha}$ and 
$\log |\lambda_n| = \frac{n+1}{2\alpha} +\delta_n$. Let us write $|\lambda_n| = \gamma_m e^{\dl}$
with $|\dl| \lesssim 1$.  Hence $\exp({\frac{n+1}{2\alpha}+\delta_n})=\exp({\frac{m+1}{2\alpha}+\dl})$ and so 
\begin{equation}
\label{delta1}
\frac{n-m}{2\alpha}=-\delta_n+\dl.
\end{equation}
Without loss of generality, let $\delta_n\to +\infty$. 
Moreover, we may assume that $\dl > 1$. Otherwise, we may 
replace $m$ by $m-m_0$ for some sufficiently large $m_0$ that can be chosen to be independent 
of $k$. (If $\delta_n\to-\infty$ then we may assume that $\dl>1$. Otherwise, we may 
replace $m$ by $m+m_0$ for some sufficiently large $m_0$)

Then by \eqref{delta1} we have
$$
\begin{aligned}
\log|A_{n,m}| & = 
n\frac{m+1}{2\alpha} - n \Big(\frac{n+1}{2\alpha} +\delta_n\Big) -
\frac{(m+1)^2}{4\alpha} + \alpha\Big(\frac{n+1}{2\alpha} +\delta_n\Big)^2  + O(1)\\
& = - \frac{(n-m)^2}{4\alpha} +\frac{n-m}{2\alpha}+\delta_n+\alpha\delta_n^2+O(1)\\
&=2\alpha\delta_n\delta_{mn}+O(1)
\end{aligned}
$$

Thus, $|A_{n,m}| \gtrsim e^{2\alpha\delta_n\dl}$. Repeating this estimate for each $k$
(note that all asymptotic estimates $\asymp$ and $\gtrsim$ hold uniformly
with respect to $k$) we conclude that
$\lim_{k\to\infty} |A_{n_k,m_k}| = \infty$. Since $A_{m,n} = (g_{\lambda_n},
\K_{\gamma_m})$, we conclude that $\|g_{\lambda_{n_k}}\|_{\varphi,2}
\to \infty$, $k\to\infty$, and so $\{\K_\lambda\}_{\lambda\in \Lambda}$ 
is not uniformly minimal (in particular, it is not a Riesz basis).
\\

(c) Given $N$, set
\begin{equation}
\label{epd}
\sup_n\frac{1}{N}\Big|\sum_{k=n+1}^{n+N}\delta_k\Big|= \frac{1}{4\alpha}+\eps_N.
\end{equation}
Since we already  have shown that $(\delta_n)$ is bounded, 
the sequence $(\eps_N)$ is bounded. If for some $N$ we have $\eps_N<0$, then (c)
is proved. Assuming the converse, we have $\eps_N \ge 0$.

Replacing, if necessary, $\Lambda$ by $e^{i\theta}\Lambda$ we may assume
that $\dist (\gamma_m, \Lambda) \asymp \gamma_m$.

Suppose first that there is a subsequence $(N_l)$ 
such that $\eps_{N_l}$ is bounded from below by $\eps>0$.
Then there exists $n_l$ such that
$$
 \frac{1}{N_l}\Big|\sum_{k=n_l+1}^{n_l+N_l}\delta_k\Big|\ge \frac{1}{4\alpha}+\eps_{N_l}/2
 \ge \frac{1}{4\alpha}+\eps/2. 
$$
It follows from \eqref{matrix} that the sum 
$|A_{n_l,n_l+N_l}|+|A_{n_l+N_l,n_l}| $ is unbounded and so {\it a fortiori} the
matrix $[A_{n,m}]$. 

Suppose now that $(\eps_N)$ is a sequence of positive numbers tending to zero.
Then for every $N$ there exists $n_N$ such that 
$$
 \Big|\sum_{k=n_N+1}^{n_N+N}\delta_k\Big|\geq N\Big(\frac{1}{4\alpha}+\eps_N \Big)-1
$$
By definition
$$
 \Big|\sum_{k=n_N+1}^{n_N+K}\delta_k\Big|\le K\Big(\frac{1}{4\alpha}+\eps_K\Big)
$$
for every $1\le K\le N$. 
Then
\begin{equation}
\label{bab3}
 \Big|\sum_{k=n_N+K+1}^{n_N+N}\delta_k\Big| \geq \frac{N-K}{4\alpha}
 +N\eps_N-K\eps_K-1. 
\end{equation}
Two cases may occur.

Case (1): If $(N\eps_N)$ has a subsequence which 
tends to $+\infty$, then for fixed $K$ and $N$ in this subsequence,
the expression $N\eps_N-K\eps_K$ is positive and unbounded. Thus, 
again by \eqref{matrix}, the sum of matrix entries $|A_{M+K,M+N}|+|A_{M+N,M+K}|$ 
is unbounded and so the matrix $[A_{n,m}]$ can not define a bounded operator 
in $\ell^2$.

Case (2): Assume that the sequence $(N\eps_N)$ is bounded.
If there exist arbitrarily large $N$ such that 
$\sum_{k=n_N +1}^{n_N +N} \delta_k >0$ then for these values of $N$ and 
every $1\le K\le N$ we have by \eqref{bab3} 
$$
 \sum_{k=n_N+K+1}^{n_N+N}\delta_k \geq \frac{N-K}{4\alpha}
 +N\eps_N-K\eps_K-1. 
$$ 
Hence, by \eqref{matrix},
$$
 |A_{n_N+N, n_N+K+1}|\asymp \exp\bigg( \sum_{k=n_N+K+1}^{n_N+N} \delta_k 
 - \frac{N-K}{4\alpha}\bigg)\asymp 1.
$$
Analogously, if there exist arbitrarily large $N$ such that 
$\sum_{k=n_N +1}^{n_N +N} \delta_k <0$, then by \eqref{bab3} for
$1\le K\le N$ we have  
$$
\sum_{k= n_N+1}^{n_N +N - K} \delta_k \le -\frac{N-K}{4\alpha} +O(1), 
$$                                                                
and so, by \eqref{matrix}, 
$$
|A_{n_N+1, n_N + N - K}| \asymp \exp\bigg( -\frac{N-K}{4\alpha} - 
\sum_{k= n_N+1}^{n_N +N - K} \delta_k \bigg) \asymp 1.
$$
In each of these two situations 
we conclude that the matrix  $[A_{n,m}]$ has an increasing number 
of entries in one line which are bounded away from zero, 
and, thus, cannot define a bounded operator.  
\qed

\begin{rem}
\label{bsm}
{\rm Now we compare in more detail Theorem \ref{BaseRiesz}  
with the results of Belov, Mengestie and Seip \cite{BMS}.
By Proposition  \ref{asdb} (see formula \eqref{1}) 
any function $f \in \ahd$ may be represented as 
$$
f(z)  = G(z) \sum_{\gamma \in \Gamma} \frac{c_\gamma |\gamma|^{1/2}}{z-\gamma}, 
$$
where $(c_\gamma)_{\gamma \in \Gamma}  
\in \ell^2$ and $\|f\|_{\varphi, 2} \asymp \|(c_\gamma)\|_{\ell^2}$.
Thus, the space $\ahd$ is a special case of the spaces $\mathcal{H}(\Gamma, v)$
introduced in \cite[Section 2]{BMS} and which consist of entire functions 
of the form $\displaystyle f(z) = G(z) \sum_{\gamma \in \Gamma} \frac{c_\gamma v_\gamma^{1/2}}{z-\gamma}$. 
The space $\ahd$ corresponds to the choice of weights
$v_n = |\gamma_n|$ (we write $v_n$ in place of $v_{\gamma_n}$). 

In \cite[Theorems 5.1, 5.2]{BMS} a description of Riesz bases in the spaces
$\mathcal{H}(\Gamma, v)$ was obtained under two additional restrictions:
\begin{equation}
\label{14}
\text{either} \quad  v_{n} = o\Big(\sum_{k<n} v_k\Big) \qquad or \quad 
\frac{v_n}{|\gamma_n|^2} = o\Big(\sum_{k>n} \frac{v_k}{|\gamma_k|^2}\Big)
\end{equation}
as $n\to \infty$. The criteria in \cite{BMS} are also of the perturbative nature
and are very close to our result. However,  for the choice of the weight 
$v_n = |\gamma_n|$ neither of the conditions in \eqref{14} is satisfied. Thus, the results
stated in \cite{BMS} do not cover Theorem \ref{BaseRiesz}. It seems very probable
that Theorem \ref{BaseRiesz} can be proved using the powerful methods developed 
in \cite{BMS} (in particular, Theorem 1.3), however our feeling is that such a proof  will not be shorter  and more transparent than the one presented in this section. }                                                                   
\end{rem}


\section{Complete interpolating sequences in $\ahc$} 
\label{section8}

In this section we prove Theorem \ref{infin}. The proof is in many ways
similar to the proof of Theorem \ref{BaseRiesz}. Put
$$
\widetilde {\Gamma} = \Gamma\cup\{\tilde \gamma\} = \{\gamma_n\}_{n\ge 0} 
\cup\{\tilde \gamma\},
$$
where $\Gamma = \Gamma_\alpha=\{e^{\frac{n+1}{2\alpha}}e^{i\theta_n}\}_{n\ge 0}$ 
is our reference sequence \eqref{bls} and $\tilde \gamma \notin \Gamma$.

\begin{prop}
\label{gam-tilda}
The sequence $\widetilde {\Gamma}$ is a 
complete interpolating sequence for $\ahc$ for any $\tilde \gamma \in \CC\setminus\Gamma$.
\end{prop}

\begin{proof}
Let $\widetilde{G}(z) = \prod_{\gamma \in \widetilde{\Gamma}} 
(1-z/\gamma)$, that is $\widetilde{G}(z) = (1-z/\tilde \gamma)G(z)$ where $G$ is the 
associated function of Lemma \ref{FonctionBL} vanishing exactly on $\Gamma$
(without loss of generality assume that $\tilde \gamma \ne 0$).
For every sequence $v=(v_{\gamma})_{\gamma
\in\widetilde {\Gamma}}$ with $\|v\|_{\infty,\varphi, \widetilde{\Gamma}} <\infty$, 
we construct the corresponding interpolating function
and estimate its norm:
$$
 f_v(z)  = \sum_{\gamma\in\widetilde {\Gamma}}  
 v_{\gamma}\frac{\widetilde{G}(z)}{\widetilde{G}'(\gamma) (z-\gamma)}.
$$
By the estimates of Lemma \ref{FonctionBL} we have, for any $z\in \C$,
\begin{eqnarray*}
 |f_v(z)| & \lesssim & 
 \sum_{\gamma\in \widetilde {\Gamma}}|v_{\gamma}| e^{\varphi(z)}
 \frac{\dist(z,\widetilde {\Gamma})}{1+|z|^{3/2}}
 \cdot \frac{|z|}{\gamma} \cdot \frac{1+|\gamma|^{3/2}}{e^{\varphi(\gamma)}|z-\gamma|}\\
 &\lesssim & 
 e^{\varphi(z)}\|v\|_{\infty,\varphi, \widetilde{\Gamma}}\sum_{\gamma\in \widetilde {\Gamma}}
 \frac{\dist(z,\widetilde {\Gamma})}{|z-\gamma|}\frac{1+\gamma^{1/2}}{1+|z|^{1/2}}
 \lesssim  \|v\|_{\varphi,\infty, \widetilde{\Gamma}} e^{\varphi(z)},
\end{eqnarray*}
and we deduce that $\widetilde {\Gamma}$ is an interpolating sequence for
which we can construct a linear operator of interpolation.

We now show that $\widetilde {\Gamma}$ is a uniqueness sequence.
For this, let $f\in \ahc$ vanish on $\widetilde {\Gamma}$. 
Consider the holomorphic function $g=f/\widetilde{G}$.
Then, again by Lemma \ref{FonctionBL},
$$
|g(z)|\lesssim\frac{1+|z|^{1/2}}{\dist(z,\widetilde{\Gamma})}
$$
and, by the maximum modulus principle, $g=0$.

As a conclusion the sequence $\widetilde {\Gamma}$ is an interpolating sequence
which is also a uniqueness sequence and thus a sampling sequence. 
\end{proof}

\begin{proof}[Proof of Theorem \ref{infin}]
Let $\Lambda = (\lambda_n)_{n\ge 0}$ be a sequence of complex numbers
tending to infinity, with $|\lambda_n| \le |\lambda_{n+1}|$. 
As before, we write $\lambda_n=\gamma_n e^{\delta_n} e^{i\theta_n}$, 
$\theta_n\in \RR$, $n\ge 0$. Now let $\wl = \Lambda \cup \{\tilde \lambda\}$.
We can also write $\wl = (\lambda_n)_{n\ge -1}$ with $\lambda_{-1} =\tilde \lambda$,
and analogously for $\widetilde{\Gamma}$.

To prove Theorem \ref{infin} we need to show that $\wl$
is a complete interpolating sequence for $\ahc$ if and only if 
$\Lambda$ is complete interpolating for $\ahd$, that is, $\Lambda$ satisfies 
the conditions (a)--(c) of Theorem \ref{BaseRiesz}.
\\

{\it $\Lambda$ is complete interpolating for $\ahd$  $\Longrightarrow$
$\wl$ is a complete interpolating sequence for $\ahc$.}
\smallskip

By Theorem \ref{BaseRiesz}, $\Lambda$ satisfies the conditions (a)--(c). Hence,
the infinite product $\wf(z) = \prod_{\lambda\in \wl} (1-z/\lambda)$ converges
uniformly on compact sets and, by \eqref{psi-estim}, there exists $\eta>0$ such that
$$
 e^{\varphi(z)} \frac{\dist(z,\wl)}{(1+|z|)^{1-\eta}}\lesssim 
 |\wf(z)|\lesssim e^{\varphi(z)} \frac{\dist(z,\wl)}{(1+|z|)^{\eta}}, 
 \qquad z\in\C.
$$     
Thus, $\wf \notin \ahc$, while $\wf/(\cdot -\lambda) \in \ahc$, $\lambda \in \wl$. 
Also, if $\wf g\in \ahc$, then $|g(z)| \lesssim |z|^{1-\eta}/\dist(z, \wl)$, 
whence $g\equiv 0$. Thus, $\wl$ is a uniqueness set for $\ahc$.

It remains to show that $\wl$ is an interpolating sequence for $\ahc$. Let 
$(v_n)_{n\ge 0}$ be a finite sequence and put
$$
f_v(z) = \sum_{n} v_n \frac{\wf(z)}{\wf'(\lambda_n)(z-\lambda_n)}.
$$
Since we already know that $\tilde{\Gamma}$ is a sampling sequence for
$\ahc$, we have
\begin{equation}
\label{bnm}
\begin{aligned}
\|f_v\|_{\varphi, \infty} & \asymp \sup_{\gamma_m\in \widetilde{\Gamma}}
e^{-\phi(\gamma_m)}\bigg| \sum_n v_n 
\frac{\wf(\gamma_m)}{\wf'(\lambda_n)(\gamma_m-\lambda_n)} \bigg| \\
& = 
\sup_{\gamma_m\in \widetilde{\Gamma}}
\Big| \sum_n v_n e^{-\phi(\lambda_n)} B_{n,m} \Big|,
\end{aligned}
\end{equation}
where
$$
B_{n,m} =  e^{\phi(\lambda_n)-\phi(\gamma_m)}
\frac{\wf(\gamma_m)}{\wf'(\lambda_n)(\gamma_m-\lambda_n)}.
$$
Note that $|\wf(\gamma_m)| \asymp \gamma_m|F(\gamma_m)|$ 
and $|\wf'(\lambda_n)| \asymp |\lambda_nF'(\lambda_n)|$, and so 
$$
|B_{n, m}| = e^{\phi(\lambda_n)-\phi(\gamma_m)} |A_{n, m}|\cdot\frac{\gamma_m}
{|\lambda_n|} \cdot 
\frac{\|\bk_{\gamma_m}\|_{\varphi,2}}{\|\bk_{\lambda_n}\|_{\varphi,2}} 
\asymp |A_{n, m}|,
$$
where $A_{n,m}$ are defined in the previous section.
By \eqref{matrix} and (c),
$|A_{n,m}| \lesssim \exp (- \delta|n-m|)$ for some $\delta >0$.
Hence, 
$$
\|f_v\|_{\phi, \infty} \le \|v\|_{\phi, \infty, \wl} 
\sup_{\gamma_m\in \widetilde{\Gamma}} \sum_{\lambda_n\in\wl} |B_{n,m}|\lesssim
\|v\|_{\phi, \infty, \wl},
$$ 
so that an interpolating function $f_v$ exists for every finitely supported sequence $v$ with
uniform control depending on $\|v\|_{\phi, \infty, \wl}$.
It remains to apply a normal family argument to show that such an interpolating function
$f_v$ exists for arbitrary $v$ with $\|v\|_{\phi, \infty, \wl} < \infty$. 
\\


{\it $\wl$ is a complete interpolating sequence for $\ahc$ $\Longrightarrow$
$\Lambda$ is complete interpolating for $\ahd$.}
\smallskip

We need to show that the sequence $\Lambda$ satisfies the conditions (a)--(c).
By Corollary~\ref{lemma2x}, $\Lambda$ is $d_\rho$-separated. Also, if 
the sequence $(\delta_n)$ is unbounded, then, as in the proof of necessity part 
of Theorem \ref{BaseRiesz},  there exists a subsequence 
$|A_{n_k, m_k}| \to \infty$. Since $|A_{n_k, m_k}| \asymp |B_{n_k, m_k}|$, 
it follows from \eqref{bnm} that the interpolation operator is unbounded 
(choose the data $v_{\lambda} = e^{\phi(\lambda)}$ for $\lambda = \lambda_{n_k}$
and $v_\lambda = 0$ otherwise).

It remains to prove (c). As in Section \ref{section7} define $\eps_N$ 
by \eqref{epd}. If the sequence $N\eps_N$ is unbounded then again
there exists a subsequence $|A_{n_k, m_k}| \to \infty$. Finally, if 
the sequence $N\eps_N$ is bounded, then, analogously to the proof of (c) 
in Section \ref{section7} we can show 
that there exist arbitrary large $N$, $m_N$ and $n_N$ such that for 
$1\le K\le N$ we have
$|B_{n_N+K, m_N}| \asymp |A_{n_N+K, m_N}| \gtrsim 1$.
It is clear from \eqref{bnm} that in this case the interpolation operator
is unbounded (choose the data $v_n e^{-\phi(\lambda_n)} \in \ell^\infty$
such that $v_n e^{-\phi(\lambda_n)} B_{n, m_N} \asymp |B_{n,m_N}|$ for
$n_N+1 \le n \le n_N+N$).
\end{proof}


\section{Proof of the density results}
\label{section4}

\subsection{Sufficient conditions}

First we deduce the sufficient conditions of Theorems \ref{thm1}--\ref{thm-int} 
from Theorems \ref{BaseRiesz} and \ref{infin}.

These conditions follow immediately from Theorem \ref{denscomp} which we recall
and prove here.

\begin{thm*}
 Let $\varphi(r) = \alpha(\log^+r)^2$, let $p=2, \infty$, and let $\Lambda$ be 
 a $d_\rho$-separated sequence. Then

 {\rm (i)} if $D^+(\Lambda) < 2\alpha$, then $\Lambda$ is a subset of some complete 
 interpolating sequence in $\ahp$\textup;

 {\rm (ii)} if  $D^-(\Lambda) > 2\alpha$, then $\Lambda$ contains a complete 
 interpolating sequence in $\ahp$. 
\end{thm*}

We give a proof for the case $p=2$; the proof for the case $p=\infty$
is completely analogous. To simplify the notations, we choose  $\alpha=1/2$. 
Recall that with our choice of $\alpha$ the set $\Gamma = \{\gamma_n\} = 
\{e^{n}\}_{n\in \mathbb{N}}$ 
becomes a complete interpolating sequence for the space.

\begin{proof}[Proof of (i)]
%
It follows from the condition $D^+(\Lambda) < 1$ 
that for sufficiently large $M>0$, every annulus
$$
A_m = \big\{z: e^{Mm+\frac{1}{2}} < |z| < e^{M(m+1)+\frac{1}{2}}\big\}, \qquad m\ge 0,
$$
contains at most $M-1$ points from $\Lambda$. Fix such $M$. Furthermore, there
exists an $\eta>0$ such that each $A_m$ contains an annulus $B_m$ of width $\eta$
which contains no points of $\Lambda$.

Our goal is to add some sequence $\Lambda'$ to $\Lambda$ so that 
the new sequence $\Lambda\cup \Lambda'$ could be written as 
$\gamma_n e^{\delta_n}e^{i\theta_n}$ and for some $N$ we would have  
\begin{equation}
\label{riesz}
\sup_n\frac{1}{N}\Big|\sum_{k=n}^{n+N}\delta_k\Big|\leq \delta<\frac{1}{4\alpha}=\frac{1}{2}.
\end{equation}

Let us denote the points from $\Lambda \cap A_m$ by
$\lambda_{1}^m, \dots, \lambda_{l_m}^m$ 
(we of course assume that $\lambda_l$ are ordered so that the modulus is nondecreasing), 
and let us associate with each of them some point from 
$\Gamma \cap A_m$. E.g., let us write
$$
\lambda_l^m = e^{Mm+ l }e^{\delta_{Mm+l}} e^{i \theta_{Mm+ l}}, \qquad 
1 \le l \le l_m.
$$
In each annulus $A_m$ we still have at least one point
from $\Gamma \cap A_m$ to which nothing is associated.

We now take a large number $N$ (the choice will be specified later) and consider
the groups of the annuli $A_m$, namely put
$$
\tilde A_k = \bigcup\limits_{m=kN+1}^{kN+N} A_m, \qquad k\ge 0. 
$$
Now in the whole group of annuli $\tilde A_k$ there are at least 
$N$ free points of 
$\Gamma$ to which we need to assign some element of the sequence $\Lambda'$ 
that we want to construct. We will do this in such a way that for any $k$
we have 
\begin{equation}
\label{r2}
\bigg|\sum_{n= (kN+1)M+1}^{(kN+N+1)M} \delta_n  \bigg| \le CM,
\end{equation}
for some absolute constant $C$ whence for sufficiently large $N$, \eqref{riesz}
will be satisfied. Thus, from now on, $k$ will be fixed.

We use an idea from the paper \cite{seip} by Seip.
The points of $\Lambda'\cap \tilde A_m$ 
will be chosen within the annuli $B_m$ (of the width $\eta$).
Note that we can even put all missing points in one annulus $B_m$, if we want,
and still have $\rho$-separation, 
but, of course, the separation constant will depend on $\eta$, 
$M$ and $N$ and may be rather 
small. Let us consider all possible sequences $\Lambda' 
\subset \cup_{m=kN+1}^{kN+N} B_m$ with separation constants 
uniformly bounded away from zero, and let us write
the elements of $\Lambda\cup \Lambda'$ as $\gamma_n e^{\delta_n}e^{i\theta_n}$.
Note that for any $m$ and $n= Mm +1, \dots, Mm + l_m$ 
the values $\delta_n$ are already fixed. 
Moreover, since for these $n$ the corresponding $\lambda$-s are in the same annulus 
$A_m$ we have $|\delta_n| \le M$, whence
$$
-M^2 N \le \sum_{m=kN+1}^{kN+N} \sum_{n = Mm +1}^{Mm +l_m} \delta_n \le M^2N.
$$

Now assume that we chose all the points of $\Lambda'$  in the annulus $B_{kN+1}$
(the smallest of all $B_m$ in our group). Then for 
$$
kNM+jM+ l_{kN+j} +1 <  n  \le kNM+(j+1)M, \qquad 2\le j \le N-1,
$$
we have
$$
\delta_n \le -(j-1)M,
$$
whence (using the fact that we have at least $N$ free indices in each $A_m$)
$$
\sum_{j=2}^{N-1} \ \sum_{n= kNM+jM + l_{kN+j} +1}^{kNM+(j+1)M} \delta_n 
\le - \sum_{j=2}^{N-1} (j-1) MN \le -\frac{MN^2}{3},
$$
when $N$ is sufficiently large. Thus,
with this choice of $\Lambda'$ we have
$$
\sum_{n= (kN+1)M+1}^{(kN+N+1)M} \delta_n 
\le  -\frac{M N^2}{3} + O(M^2 N) <0, 
$$
if $N\gg M$.

Analogously, if we choose all the points of $\Lambda'$  in the annulus $B_{kN+N}$
(the largest of all $B_m$ in our group), we will have 
$\delta_n \ge (N-2-j)M$ for $kNM+jM+ l_{kN+j} +1 <  n  \le kNM+(j+1)M$, 
$0\le j\le N-3$, whence
$$
\sum_{n= (kN+1)M+1}^{(kN+N+1)M} \delta_n 
\ge  \frac{M N^2}{3} - O(M^2 N) >0. 
$$

Finally, note that if two choices of $\Lambda'$ coincide up  to one  point
which is in some $B_m$ for one choice and which is in $B_{m+1}$ for the other choice, 
then the corresponding sums
$$
\sum_{n= (kN+1)M+1}^{(kN+N+1)M} \delta_n 
$$
considered for these two choices of $\Lambda'$
will differ by at most $2M$. Since the two configurations of $\Lambda'$
described above may be obtained from the other by changing only one point 
and moving it to a neighboring annulus $B_n$, we conclude 
that there exists some intermediate choice of $\Lambda'$
with the property \eqref{r2} (with $C=2$).
\end{proof}
\begin{proof}[Proof of (ii)]
 The idea is the same and so we may omit some details.
Let $M, N$ be as above, but now we assume that each $A_m$ contains at least 
$M+1$ points for some fixed $M$. Let us assume that $N\gg M$
and choose $j_0 \in \mathbb{N}$ so that $3j_0 M <N \le 3(j_0+1)M$.

For $j_0\le j \le N-j_0$
and $kNM +jM +1 \le n \le kNM +(j+1)M$ we choose in an arbitrary way $\lambda_n \in 
\Lambda\cap A_j$ and write them as
$$
\gamma_n e^{\delta_n}e^{i\theta_n}.
$$ 
Then $|\delta_n|\le M$ and
$$
-(N-2)M^2 \le \sum_{j=j_0}^{N-j_0} \sum_{n= kNM +jM +1}^{kNM +(j+1)M} \delta_n\le  (N-2)M^2.
$$
Note that we did not assign any point from $\Lambda$ to $n$-s  in the first and in 
the last interval, namely, for 
$kNM +1 \le n \le kNM +j_0M$
and for $kNM +(N-j_0)M +1 \le n \le kNM +NM$.

Recall that we still have $N$ free points of $\Lambda$ in each $A_m$.
Now consider two choices of $\lambda_n$ for these values of $n$. 
For the first choice let us assign some points 
$\lambda_n \in \Lambda\cap A_j$ to $kNM +jM +1 \le n \le kNM +(j+1)M$ and 
$N-j_0 \le j\le N-1$. However, for
$kNM +1 \le n \le kNM +j_0M$ let us choose $j_0M$ points $\lambda_n$ 
in $\cup_{j>2N/3} A_j$. This is possible, since we have at least $N/3> j_0M$ 
free points from $\Lambda$ in $\cup_{j>2N/3} A_j$. Then, 
for $kNM +1 \le n \le kNM +j_0 M$, we have
$$
\delta_n \le - \frac{MN}{3},
$$
and hence,
$$
\begin{aligned}
\sum_{n= (kN+1)M+1}^{(kN+N+1)M} \delta_n  
&  =  \sum_{j=0}^{j_0-1} \ \sum_{n=kNM +jM + 1}^{kNM +(j+1) M} \delta_n + 
\sum_{j=j_0}^{N-1} \ \sum_{n=kNM +jM + 1}^{kNM +(j+1) M} \delta_n \\
& \le - \frac{j_0 M^2 N}{3} + O(M^2 N) <0.
\end{aligned}
$$ 

Analogously, choosing the points $\lambda_n \in \Lambda\cap A_j$ 
for $kNM +jM +1 \le n \le kNM +(j+1)M$ and $0\le j\le j_0-1$, and
taking $\lambda_n$ in $\cup_{j<N/3} A_j$ for
$kNM +(N-j_0)M +1 \le n \le kNM +NM$, we see that the 
corresponding sum of $\delta_n$ is positive. 

The proof is completed as in (i): each configuration $\Lambda_n$ may 
be obtained from the other by changing exactly on point at each step, 
and, moreover, these points can be chosen at the distance (with respect 
to the logarithm) at most $2M$. Thus, the corresponding sum will be at most $4M$
for some choice of $\{\lambda_n\} \subset \Lambda$.
\end{proof}


\subsection{Necessary conditions for sampling/interpolation, $p=2,\infty$, Theorems 
\ref{thm1}, \ref{thm2} and \ref{thm-int}}

To obtain the necessary conditions  for the sequence to be  sampling/interpolating, 
we use the technique developed by Ramanathan and Steger \cite{OS,MMO}.  
We follow the scheme of proof proposed in \cite[Lemma 40]{MMO} and 
concentrate mainly on the places where the proofs differ.

\begin{lem}\label{InterSamplRS}Let $\varepsilon>0$. Assume $\Lambda $ to be interpolating for   $\mathcal{F}_{(1-\varepsilon)\varphi}^{p}$, $p=2,\infty$,   and $\cS$ to be sampling for 
 $\mathcal{F}_{\varphi}^{2}$  and $d_\rho$-separated. Then  for small $\delta>0$ we have  
for sufficiently big $R$,
$$
(1-\delta^2)\Card (\Lambda\cap\cA(x,Rx))\leq  \Card (\cS\cap\cA(\delta x,Rx/\delta)).
$$
\end{lem}
\begin{proof}
Let $p=2$. Since $\Lambda $ is interpolating for 
$\mathcal{F}_{(1-\varepsilon)\varphi}^{2}$, for every $\lambda\in \Lambda $ 
there exists $f_{\lambda}\in \mathcal{F}_{(1-\varepsilon)\varphi}^{2}$, such that $f_\lambda(\lambda)=1$, $f_\lambda|\Lambda \backslash\{\lambda\}=0$ and 
$ \|f_\lambda\|_{(1-\varepsilon)\varphi,2}\lesssim |\lambda|e^{-(1-\varepsilon)\varphi(\lambda)}$. 
By \eqref{kernel-estimate},  
\begin{eqnarray}\label{estimflambda}
 |f_\lambda(z)|=|\langle f_{\lambda},\bk_z\rangle|
 \lesssim e^{(1-\varepsilon)(\varphi(z)-\varphi(\lambda))}|\lambda|/(1+|z|).
\end{eqnarray}
Let $G$ be the function from Lemma \ref{FonctionBL} associated to 
$\Gamma:=\Gamma_{\varepsilon\alpha}=
\{e^\frac{n+1}{2\varepsilon \alpha}e^{i\theta_n}\text{ : } n\geq 0\}$, and let $\gamma_\lambda\in \Gamma$  be a point such that $\dist(\lambda,\Gamma)=|\lambda-\gamma_\lambda|$. Then $|\lambda|\asymp\gamma_\lambda$. With an appropriate choice of $\theta_n$ we can
assume that $d_{\rho}(\Lambda,\Gamma)>0$. Define
$$
\kappa(z,\lambda):=
 \left\{\begin{array}{ll}
  f_\lambda(z)\frac{\dst G(z)}{\dst z-\gamma_\lambda}
 \frac{\dst \lambda-\gamma_\lambda}{\dst G(\lambda)}
 \frac{\dst z}{\dst \lambda}
 \|\bk_\lambda\|_{\varphi,2} &\text{if }z\in \C\setminus\{\gamma_\lambda\},\\
  {G'(\lambda)}\frac{\dst \lambda-\gamma_\lambda}{\dst G(\lambda)}\|\bk_\lambda\|_{\varphi,2} &\text{if }z=\gamma_\lambda.
 \end{array}
 \right.
$$
By construction, $\kappa(\cdot,\lambda)\in \ahd$, $\lambda\in \Lambda$, and the system $\{\kappa(\cdot,\lambda)\}_{\lambda\in \Lambda}$ is 
biorthogonal to $\{\K_\lambda\}_{\lambda\in \Lambda }$.
Moreover $\|\kappa(\cdot,\lambda)\|_{\varphi,2}$ is uniformly bounded.
To verify this, it suffices to estimate this 
norm on the Borichev--Lyubarskii sampling sequence $\Gamma_{\alpha}$ with 
$\dist(\gamma,\Gamma_{\varepsilon \alpha})\asymp 
|\gamma|$, $\gamma\in \Gamma_\alpha$, using 
Lemma \ref{FonctionBL} and \eqref{estimflambda}:
\begin{eqnarray}\label{estimkappa}
\lefteqn{\|\kappa(\cdot,\lambda)\|^2_{\varphi,2}
 \asymp\sum_{\gamma\in\Gamma_{\alpha}}
 |\kappa(\gamma,\lambda)|^2e^{-2\varphi(\gamma)}(1+|\gamma|^2)}\nonumber\\
 &&\asymp \sum_{\gamma\in\Gamma_{\alpha}}
 |f_{\lambda}(\gamma)|^2\frac{e^{2\varepsilon\varphi(\gamma)}\dist^2(\gamma,
 \Gamma_{\varepsilon\alpha})}{(1+|\gamma|^3)|\gamma-\gamma_{\lambda}|^2}
 \frac{|\lambda-\gamma_{\lambda}|^2(1+|\lambda|^3)}{e^{2\varepsilon\varphi(\lambda)}
 \dist^2(\lambda,\Gamma_{\varepsilon\alpha})}\times \left|\frac{\gamma}{\lambda}\right|^2
\frac{e^{2\varphi(\lambda)}}{1+|\lambda|^2}
 e^{-2\varphi(\gamma)}(1+|\gamma|^2)\nonumber\\
 &&\lesssim  \sum_{\gamma\in\Gamma_{\alpha}}
 \frac{\dist^2(\gamma,\Gamma_{\varepsilon\alpha})}{|\gamma-\gamma_{\lambda}|^2}
 \frac{1+|\lambda|}{1+|\gamma|}\lesssim  \sum_{\gamma\in\Gamma_{\alpha},|\gamma|\le
 |\lambda|}
 \frac{1+|\gamma|}{1+|\lambda|}+
  \sum_{\gamma\in\Gamma_{\alpha},|\gamma|>|\lambda|}
 \frac{1+|\lambda|}{1+|\gamma|}.
 \end{eqnarray}
Both sums are majorized by the sum of a geometric progression, and, 
hence, are uniformly bounded.

Let now  $\{\widetilde{k}(\cdot,s)_{s\in \cS }\}$ be the dual frame 
(see, for example, \cite{ccc}) for  
$\{\bk_s/\|\bk_s\|_{\varphi,2}\}_{s\in \mathcal{S}}$ in $\mathcal{F}_{\varphi}^{2}$.  Consider the  following
finite dimensional subspaces of $\mathcal{F}_{\varphi}^{2}$: 
$$W_{\cS}=\{\widetilde{k}(\cdot,s)\text{ : }s\in \cS\cap\cA(\delta x,Rx/\delta) \}\quad \text{ and } \quad  W_{\Lambda }=\{\kappa(\cdot,\lambda)\text{ : }\lambda\in \Lambda \cap\cA(x,Rx) \}. $$
We define
$P_{\cS} $ and $P_\Lambda $ as the orthogonal projections of     $\mathcal{F}_{\varphi}^{2}$  onto $W_{\cS} $ and $W_\Lambda $, respectively.
Consider the operator $T = P_\Lambda  P_{\cS} $ defined from $W_\Lambda $ to $W_\Lambda $. 
Clearly
\[
 \textrm{tr}(T)  
 \leq \textrm{rank}P_{\cS} \leq  \Card (\cS \cap\cA(\delta x,Rx/\delta)).
\]
Since $(P_{\Lambda}\K_{\lambda})_{\lambda}$ is biorthogonal to
$(\kappa(\cdot,\mu))_{\mu\in\Lambda\cap \cA(x,Rx)}$ in $W_{\Lambda}$,
we have
$$\textrm{tr}(T) =\sum_{\lambda\in \Lambda \cap\cA(x,Rx)} \langle T\kappa(\cdot,\lambda),P_{\Lambda}\K_{\lambda} \rangle.
$$
Hence, since $\langle T\kappa(\cdot,\lambda),P_{\Lambda}\K_{\lambda}\rangle
=\langle\kappa(\cdot,\lambda)+(P_{\mathcal{S}}-Id)\kappa(\cdot,\lambda),P_{\Lambda}\K_{\lambda}\rangle$, we also get
\[
  \textrm{tr}(T)  \ge \Card (\Lambda \cap\cA(x,Rx))[1-\sup_{\lambda} \|P_{\cS} (\kappa(\cdot,\lambda))-\kappa(\cdot,\lambda)\|_{\varphi,2} ] 
\]
It remains to verify that $\|P_{\cS} (\kappa(\cdot,\lambda))-\kappa(\cdot,\lambda)\|_{\varphi,2}$
are small for sufficiently small $\delta$ independently of $\lambda$. By Lemma  \ref{FonctionBL} we have 
\begin{equation}
\label{kappa}
\frac{|\kappa(s,\lambda)|^2}{\|\bk_s\|^{2}_{\varphi,2}}\lesssim\frac{|\lambda|}{1+|s|}
\frac{\dist (s,\Gamma_{\varepsilon\alpha})^2}{|s-\gamma_\lambda|^2}.
\end{equation}
Since $\cS$ is sampling and $d_\rho$-separated, and since $\lambda\in \cA(x,Rx)$, we have,
using also $|\lambda|\asymp |\gamma_{\lambda}|$,   
\begin{eqnarray*}
\|P_{\cS} (\kappa(\cdot,\lambda))-\kappa(\cdot,\lambda)\|_{\varphi,2}&\lesssim& \sum_{s\notin\cA(\delta x,Rx/\delta )}\frac{|\kappa(s,\lambda)|^2}{\|\bk_s\|^{2}_{\varphi,2}}\\
&\lesssim&
\sum_{s\notin\cA(\delta x,Rx/\delta )}\frac{|\gamma_\lambda|}{1+|s|}
\frac{\dist (s,\Gamma_{\varepsilon\alpha})^2}{|s-\gamma_\lambda|^2} \\
&\lesssim&  \sum_{|s|\leq |\gamma_\lambda|,\, s\notin\cA(\delta x,Rx/\delta ) }\frac{|s|}{|\gamma_\lambda|} +  \sum_{|s|\geq |\gamma_\lambda|,\, s\notin\cA(\delta x,Rx/\delta ) }\frac{|\gamma_\lambda|}{|s|}.
\end{eqnarray*}
For $R$ big enough, we can suppose $\gamma_{\lambda}\in \cA(x,Rx)$. Since
$\cS$ is separated, each annulus $\cA_k:=\cA(2^k,2^{k+1})$ contains
a uniformly bounded number of points of $\cS$. Let $M$ be an upper bound
of these numbers. Let $N,K$ be such that $2^N\le \delta x\le 2^{N+1}$
and $2^K\le Rx/\delta \le 2^{K+1}$. 
Then an estimate analogous to the above yields:
\begin{eqnarray*}
\sum_{|s|\leq |\gamma_\lambda|,\, s\notin\cA(\delta x,Rx/\delta ) }\frac{|s|}{|\gamma_\lambda|}
 &\lesssim& \frac{1}{|\gamma_{\lambda}|}\sum_{k=1}^N\sum_{s\in \cA_k}|s|
 \lesssim \frac{M2^N}{|\gamma_{\lambda}|}\lesssim \delta,\\
\sum_{|s|\ge |\gamma_\lambda|,\, s\notin\cA(\delta x,Rx/\delta ) }\frac{|\gamma_\lambda|}{|s|}
 &\lesssim& {|\gamma_{\lambda}|}\sum_{k\ge K}\sum_{s\in \cA_k}\frac{1}{|s|}
 \lesssim \frac{M|\gamma_{\lambda}|}{2^K}\lesssim \delta,\\
\end{eqnarray*}
and we are done.

Now let $p=\infty$. Since $\Lambda $ is interpolating for 
$\mathcal{F}_{(1-\varepsilon)\varphi}^{\infty}$, then for every $\lambda\in \Lambda $, 
there exists $f_{\lambda}\in \mathcal{F}_{(1-\varepsilon)\varphi}^{\infty}$, such that $f_\lambda(\lambda)=1$, $f_\lambda|\Lambda \backslash\{\lambda\}=0$ and 
$\|f_\lambda\|_{(1-\varepsilon)\varphi,\infty}\lesssim e^{-(1-\varepsilon)\varphi(\lambda)}$,
i.e., 
$|f_\lambda(z)|\lesssim  e^{(1-\varepsilon)(\varphi(z)-\varphi(\lambda))}$. Set 
$$\kappa(z,\lambda)= f_\lambda(z)\frac{G(z)}{z-\gamma_\lambda}\frac{\lambda-\gamma_\lambda}{G(\lambda)}\|\bk_\lambda\|_{\varphi,2}, \qquad z\in \C.$$
(Observe that we do not need the factor $z/\lambda$ in this case.)
As in \eqref{estimkappa} it can be shown that $\kappa(\cdot,\lambda)$ is in 
$\mathcal{F}_{\varphi}^{2}$ with uniformly bounded norms.
Again by Lemma \ref{FonctionBL} we have \eqref{kappa} and applying the same arguments as above
(note that $\mathcal{S}$ is again assumed sampling in $\mathcal{F}_{\varphi}^{2}$), 
we get our result.
\end{proof}

$\bullet$ Let us first deduce from this lemma the necessary condition for sampling in Theorem \ref{thm2} ($p=2$). Suppose that $\Lambda$ is a sampling sequence for  $\mathcal{F}_{\varphi}^{2}$. 
By Corollary \ref{lemma2} and Lemma \ref{lemma3},  
$\Lambda$  is a finite union of $d_\rho$-separated subsets and 
contains a $d_\rho$-separated subset $\Lambda^*$ which is also a sampling set for $\mathcal{F}_{\varphi}^{2}$.  
Let  $\Gamma=\Gamma_{(1-\varepsilon)\alpha}=
\{e^\frac{n+1}{2(1-\varepsilon )\alpha}e^{i\theta_n}\text{ : } n\geq 0\}$ for some $\varepsilon>0$. We have $D^-(\Gamma)=D^+(\Gamma)=2(1-\varepsilon)\alpha$. By \cite[Theorem 2.8]{BL}, $\Gamma$ is an interpolating sequence for  $\mathcal{F}_{(1-\varepsilon)\varphi}^{2}$ and the comparison Lemma \ref{InterSamplRS} gives $D^-(\Lambda^*)\ge 2\alpha$.  
\medskip

$\bullet$ The necessary condition for sampling in Theorem \ref{thm1} ($p=\infty$) 
follows from Corollary \ref{lemma2}, Lemma \ref{passageinfty2}, 
and the necessary condition for sampling in  Theorem \ref{thm2}.
\medskip

$\bullet$ Next we consider the necessary condition for interpolation 
in Theorem \ref{thm-int}. Consider   $\Gamma=\Gamma_{(1+\varepsilon)\alpha}=
\{e^\frac{n+1}{2(1+\varepsilon )\alpha}e^{i\theta_n}\text{ : } n\geq 0\}$ 
for some $\varepsilon>0$. We have $D^-(\Gamma)=D^+(\Gamma)=2(1+\varepsilon)\alpha$ 
and by \cite[Theorem 2.8]{BL}, $\Gamma$ is a sampling  sequence for  
$\mathcal{F}_{(1+\varepsilon)\varphi}^{2}$.  If $\Lambda$ is an interpolating sequence for  $\mathcal{F}_{\varphi}^{p}$,  $p=2,\infty$, then by Lemma \ref{InterSamplRS},  comparing the densities between interpolating and sampling sequences, we obtain $D^+(\Lambda)\le 2\alpha$.


\section{Examples} 
\label{section9}

In this section we give explicit examples of $d_\rho$-separated sequences 
of critical density which are uniqueness sets for $\ahp$ but are 
neither sampling nor interpolating.

We will also show that it is not possible to switch from a sampling
sequence to an interpolating sequence by removing a point without one being
already complete interpolating.

For the first of these examples we will need a two-sided version of the function
from Lemma \ref{FonctionBL}, which we estimate using essentially the same argument.
As before, $\phi(z) = a(\log^+|z|)^2$, $a>0$.

\begin{lem}\label{lemmeBorLyu} Let $a>0$ and consider $\Gamma^{\pm}=\{ 
-e^{\frac{n}{a}} e^{i\theta_n},e^{\frac{n}{a}} e^{i\theta_n}\}_{n\geq 1}
=:\{\gamma_n\}_{n\in\Z\setminus \{0\}}$, 
where $\theta_n$ are arbitrary real numbers. Then the infinite product
$$
E(z)=\prod_{\gamma\in \Gamma^{\pm}}\Big(1-\frac{z}{\gamma}\Big)
$$
 converges on every compact subset of $\C$, 
\begin{equation}
 \label{estimSym}
 |E(z)| \asymp e^{\phi(z)}
 \frac{\dist(z,\Gamma^{\pm})}{1+|z|^{2}},\qquad z\in \C,
 \end{equation}
and
$$
 |E'(\gamma)|\asymp \frac{e^{\phi(\gamma)}} {|\gamma|^2},
 \quad \gamma\in\Gamma^{\pm}.
$$
\end{lem}

\begin{proof}
For every $t>0$ there exists a unique $m$ such that
\[
 \frac{1}{a}\left({m}-\frac{1}{2}\right)
 \le t <\frac{1}{a}\left({m}+\frac{1}{2}\right).
\]
We consider $z$ with $|z|=e^t$. 
Then
\begin{eqnarray*}
 \log|E(z)|
 &=&\sum_{1\le |k|\le m}\log \left|1-\frac{z}{\gamma_k}\right| + O(1)\\
 &=&\sum_{1\le |k|\le m-1}\Big(t-\frac{|k|}{a}\Big)
 +\log\dist (z,\Gamma)-t+O(1)\\
 &=&2(m-1)t-\frac{2}{a}\frac{(m-1)m}{2}+\log\dist (z,\Gamma^{\pm})-t+O(1)\\
& =&at^2-2t+\log\dist(z,\Gamma^{\pm})+O(1),
\end{eqnarray*}
which gives the first estimate.
The second estimate can be immediately deduced from the first.
\end{proof}

\begin{exam}
{\rm Let $\Gamma^{\pm}$ be the sequence from Lemma \ref{lemmeBorLyu}.
In view of Theorem \ref{BaseRiesz}, this sequence cannot be complete
interpolating for $\ahd$ since the points $\gamma_{n}$, $n<0$, are too far from
the reference sequence. Clearly, the upper and lower densities of this 
sequence are equal to $2\alpha$. Let $E$ be the function of Lemma \ref{lemmeBorLyu}. Hence the functions
$g_{n}:= E/ E'(\gamma_n)(\cdot-\gamma_n)$  are in $\ahc$ and $\ahd$.

Let us prove that it is neither sampling nor interpolating for $\ahd$.

Suppose first that it is sampling. Then for every finite sequence
$v=(v_n)_{n\in\Z\setminus\{0\}}\in \ell^2_{\varphi,\Gamma^{\pm}}$
the function $f_v=\sum_n v_ng_n$ would interpolate $v$ on $\Gamma^{\pm}$, and
by the sampling property of $\Gamma^{\pm}$,
\[
 \|f_v\|_{\varphi,2}^2
 \asymp \|f_v|_{\Gamma^{\pm}}\|_{\varphi,2,\Gamma^{\pm}}^2
 =\sum_n|v_n|^2e^{-2\varphi(\gamma_n)}(1+|\gamma_n|^2) 
 =\|v\|_{\varphi,2,\Gamma^{\pm}}.
\]
In other words, the interpolation operator $v\longmapsto f_v$ would be continuous
from $\ell^2_{\varphi,\Gamma}$ to $\ahd$, and the sequence would be 
interpolating. Since it was supposed sampling it would thus be complete interpolating, and we would get a contradiction. 

Suppose next that it is interpolating.
Since the function $E$ vanishes on $\Gamma^{\pm}$ and satisfies  \eqref{estimSym},  
$\Gamma^{\pm}$ is a uniqueness sequence (see, for example, \cite[Theorem 3]{DK}). 
Since it is also interpolating, it is complete interpolating, and again we 
obtain a contradiction. }
\end{exam}

\begin{exam}
{\rm Let $\Gamma_2:=\Gamma^{\pm}\cup \{-1,1\}=\{\pm e^{n/\alpha}\}_{n\geq 0}=
\{\gamma_n\}_{n\in \Z}$ where $\Gamma^{\pm}$ is again the sequence from
Lemma \ref{lemmeBorLyu} and $\gamma_{-n}=-\gamma_{n}$. 
Then $\Gamma_2$ is neither sampling nor interpolating for $\ahc$.
Clearly, by \eqref{estimSym}, $\Gamma_2$ is a uniqueness set for $\ahc$.

Set $E_2(z)=(z^2-1)E(z)$. Then,
in view of \eqref{estimSym},
$\Gamma_2$ is a uniqueness sequence for $\ahc$ and 
$g_{\gamma}=E_2/E_2'(\gamma)(\cdot-\gamma)$ with $\gamma\in \Gamma_2$  
defines a function in $\ahc$.
Next, let  the
family of functions $(f_n)_{n\geq 0}\in \ahc$ be given by 
\[
 f_n(z)=\sum_{|\gamma|\le \gamma_n}\eps_{\gamma}g_{\gamma}(z),\qquad z\in\C,
\] 
where 
$ \eps_{\gamma}:=e^{-i\arg g_{\gamma}(z_n)}e^{\varphi(\gamma)}$ and  $z_n=i\gamma_{n+1} $. We have 
\[
 f_n(z_n)=\sum_{|\gamma|\le\gamma_ n}\eps_{\gamma}g_{\gamma}(z_n)
 =\sum_{|\gamma|\le \gamma_n} e^{\varphi(\gamma)} |g_{\gamma}(z_n)|
 \asymp e^{\varphi(z_n)}\sum_{|\gamma|\le \gamma_n}\frac{\dist (z_n,\Gamma)}{|z_n-\gamma|}
 \asymp ne^{\varphi(z_n)}.
\]
As a consequence, $\|f_n\|_{\varphi,\infty}\to\infty$, while $\|f_n|_{\Gamma_2}\|_{
\varphi,\infty,\Gamma_2}=1$. Hence the sequence is not sampling, and since it is
uniqueness it can neither be interpolating (otherwise it would be complete interpolating
and hence sampling).  }
\end{exam}

Next we give an example of a one-sided sequence of critical density which is neither
sampling nor interpolating for $\ahd$. Let $0<\delta<\frac{1}{2a}$ and let 
$$
\Lambda=\{\lambda_n\}_{n\ge 0} = \{ e^{\frac{n+1}{2a} +\delta} \}_{n\ge 0}.
$$
Then, by estimates analogous to Lemma \ref{lemmeBorLyu}, the generating function $F$
of the sequence $\Lambda$ satisfies
\begin{equation}
 \label{bbb1}
 |F(z)| \asymp e^{\phi(z)}
 \frac{\dist(z, \Lambda)}{1+|z|^{3/2 + 2a\delta}}, \qquad z\in \C.
\end{equation}
It is clear that $F\in \ahd$ if and only if $\delta >\frac{1}{4a}$. 
Moreover, by Theorem \ref{BaseRiesz}, $\Lambda$ is a complete interpolating sequence
for $\ahd$ if $0< \delta < \frac{1}{4a}$ and $\Lambda \cup\{1\}$ is   
a complete interpolating sequence
if $\frac{1}{4a} < \delta < \frac{1}{2a}$.

\begin{exam}
{\rm  If $\delta  = \frac{1}{4a}$, then $\Lambda$
is neither sampling nor interpolating for $\ahd$. 
First, it follows from \eqref{bbb1} that $\Lambda$ is a uniqueness set for $\ahd$.
Indeed, if $SF \in \ahd$ for an entire function $S$, then $S$ should be constant, but $F\notin\ahd$
(a detailed proof can be found in \cite[Theorem 3]{DK}). Hence, $\Lambda$
is not an interpolating sequence (otherwise, it would be complete interpolating which is not true).

Assume now that $\Lambda$ is sampling for $\ahd$. Then, by the stability 
result of Corollary \ref{pertp=2}, the sequence
$$
\widetilde{\Lambda}= \{ e^{\frac{n+1}{2a} +
\frac{1}{4a} + \varepsilon} \}_{n\ge 0}
$$
also is sampling for sufficiently small $\varepsilon>0$. However, it follows from \eqref{bbb1}
that $\widetilde{\Lambda}$ is the zero set of some nontrivial function in $\ahd$,
a contradiction. } 
\end{exam}

As mentioned in the introduction we 
finish this section showing that it is not possible to switch from a sampling
sequence to an interpolating sequence by removing a point without one being
already complete interpolating.

\begin{prop}\label{intsamp}
Suppose $\Lambda$ is an interpolating sequence for $\ahc$ and 
$\Lambda\cup\{\lambda\}$ is not interpolating  for a $\lambda\notin\Lambda$. Then
$\Lambda$ 
is complete interpolating for $\ahc$.
\end{prop}

\begin{proof}
The assertion of the theorem is clearly true when $\Lambda$ is complete interpolating.

Suppose now that $\Lambda$ is an interpolating sequence which is not 
a uniqueness set.
Then there exists $h$ vanishing on $\Lambda$
and such that $h(\lambda)=1$ which implies  
that $\Lambda\cup\{\lambda\}$
is interpolating.\end{proof}


\begin{lem}
If $\Lambda$ is sampling for $\ahc$ and  $\Lambda\setminus\{\lambda\}$ is a zero set for $\ahc$, then $\Lambda\setminus\{\lambda\}$ is interpolating.
\end{lem}

\begin{proof}
By assumption there exists a function $f\in\ahc$ vanishing on $\Lambda\setminus
\{\lambda\}$, and $f(\lambda)=1$. For $\mu\in\Lambda\setminus \{\lambda\}$
define the entire function
\[
 g_{\mu}(z)=
\left\{
\begin{array}{ll}
 f(z)\frac{\dst z-\lambda}{\dst z-\mu} &\text{if }z\neq \mu\\
 f'(\mu)(\mu-\lambda) &\text{if }z=\mu,
\end{array}
\right.
\]
which is in $\ahc$. Clearly $g_{\mu}$ vanishes on $\{\lambda\}\cup
\Lambda\setminus\{\mu\}$.

Pick now a finite sequence $(v_{\mu})_{\mu\neq\lambda}$ 
and set
\[
 f_v(z)=\sum_{\mu\neq \lambda}v_{\mu}\frac{g_{\mu}(z)}{g_{\mu}(\mu)}.
\]
By construction $f_v\in\ahc$ as a finite sum of functions in $\ahc$,
and $f_v$ interpolates $v_{\mu}$ in $\mu\neq \lambda$. Let
us estimate the norm of $f_v$. Observe that $f_v(\lambda)=0$ since
$g_{\mu}(\lambda)=0$ for every $\mu\neq \lambda$.
Using the fact that $\Lambda$ is sampling we have
\begin{eqnarray*}
  \|f_v\|_{\varphi,\infty}\asymp \sup_{\mu\in\Lambda}|f_v(\mu)|
 e^{-\varphi(\mu)}=\sup_{\mu\in\Lambda\setminus\{\lambda\}}|f_v(\mu)|
 e^{-\varphi(\mu)}
 =\sup_{\mu\in\Lambda\setminus\{\lambda\}}|v_{\mu}|
 e^{-\varphi(\mu)}=\|v\|_{\varphi,\infty,\Lambda\setminus\{\lambda\}}.
\end{eqnarray*}
Hence we can define a bounded interpolation operator $v\longmapsto f_v$
and $\Lambda\setminus\{\lambda\}$ is an interpolating sequence.
\end{proof}

\begin{cor}
If $\Lambda$ is sampling for $\ahc$ and  $\Lambda\setminus\{\lambda\}$ is a zero set for $\ahc$, then 
$\Lambda$ is complete interpolating.
\end{cor}

\begin{proof}
We already know from the preceding lemma that $\Lambda\setminus\{\lambda\}$ is interpolating.
Two cases may occur. Either the sequence $\Lambda$ is interpolating, in which case nothing has
to be proved, or $\Lambda$ is not interpolating. 
In the latter case, $\Lambda\setminus\{\lambda\}$ is interpolating and
$\Lambda$ is not, so that from Proposition \ref{intsamp} we conclude that
$\Lambda\setminus\{\lambda\}$ is complete interpolating which is impossible 
since $\Lambda\setminus\{\lambda\}$ is a zero set.
\end{proof}

\begin{cor}
If a sequence $\Lambda$ is interpolating for $\ahc$ and  $\Lambda\cup \{\lambda\}$ 
is sampling for $\ahc$, then either $\Lambda$ or $\Lambda \cup \{\lambda\}$ 
is complete interpolating.
\end{cor}

The reader might have noticed that these last results work in a quite general setting.
In particular, if there are no complete interpolating sequences in such a general space,
then it is not possible to switch from an interpolating sequence to a sampling 
sequence by adding a sole point.

\end{document}